\documentclass[12pt]{article}
\usepackage{mathrsfs}
\usepackage{graphicx}
\usepackage{epstopdf}
\usepackage{amsmath}
\usepackage{amsfonts}
\usepackage{amssymb, amsmath, cite}
\usepackage{cases}

\newtheorem{theorem}{Theorem}
\newtheorem{lemma}{Lemma}

\newtheorem{remark}{Remark}

\newbox\qedbox
\newenvironment{proof}{\smallskip\noindent{\bf Proof.}\hskip \labelsep}%
                        {\hfill\penalty10000\copy\qedbox\par\medskip}

\newcommand{\bfR}{{\Bbb R}}
\newcommand{\bfC}{{\Bbb C}}
\newcommand{\bfZ}{{\Bbb Z}}

\newcommand{\ii}{\text{i}}
\newcommand{\e}{\text{e}}
\newcommand{\dd}{\text{d}}
\newcommand{\Om}{\Omega}

\newcommand{\nn}{\nonumber}
\newcommand\be{\begin{equation}}
\newcommand\ee{\end{equation}}
\newcommand{\bea}{\begin{eqnarray}}
\newcommand{\eea}{\end{eqnarray}}
\newcommand\berr{\begin{eqnarray*}}
\newcommand\eerr{\end{eqnarray*}}
{\setlength\arraycolsep{2pt}
\setlength{\lineskiplimit}{2.625bp}
\setlength{\lineskip}{2.625bp}

\begin{document}

\title{The Gerdjikov-Ivanov type derivative nonlinear Schr\"odinger equation: Long-time
dynamics of nonzero boundary conditions}
\author{ Boling Guo$^{a}$,\, Nan Liu$^{b,}$\footnote{Corresponding author.}\\
$^a${\small{\em Institute of Applied Physics and Computational Mathematics,  Beijing 100088, P.R. China}} \\
$^b${\small{\em The Graduate School of China Academy of Engineering Physics, Beijing 100088, P.R. China}}\\
\setcounter{footnote}{-1}\footnote{E-mail address:  ln10475@163.com (N. Liu).} \\
}

\date{}
\maketitle

\begin{quote}
{{{\bfseries Abstract.} We consider the Gerdjikov--Ivanov type derivative nonlinear Schr\"odinger equation
\berr
\ii q_{t}+q_{xx}-\ii q^2\bar{q}_{x}+\frac{1}{2}(|q|^4-q_0^4)q=0
\eerr
on the line. The initial value $q(x,0)$ is given and satisfies the symmetric, nonzero boundary conditions at infinity, that is, $q(x,0)\rightarrow q_\pm$ as $x\rightarrow\pm\infty$, and $|q_\pm|=q_0>0$. The goal of this paper is to study the asymptotic
behavior of the solution of this initial-value problem as $t\rightarrow\infty$. The main tool is the asymptotic analysis of an associated matrix Riemann--Hilbert problem by using the steepest descent method and the so-called $g$-function mechanism. We show that the solution $q(x,t)$ of this initial value problem has a different asymptotic behavior in different regions of the $xt$-plane. In the regions $x<-2\sqrt{2}q_0^2t$ and $x>2\sqrt{2}q_0^2t$, the solution takes the form of a plane wave. In the region $-2\sqrt{2}q_0^2t<x<2\sqrt{2}q_0^2t$, the solution takes the form of a modulated elliptic wave.

}

 {\bf Keywords:} Gerdjikov--Ivanov type derivative nonlinear Schr\"odinger equation; Riemann--Hilbert problem; Long-time asymptotics; Nonlinear steepest descent method.}
\end{quote}

\section{Introduction}
\setcounter{equation}{0}

The celebrated nonlinear Schr\"odinger (NLS) equation has been recognized as a ubiquitous mathematical model among many integrable systems, which governs weakly nonlinear and dispersive wave packets in one-dimensional physical systems. It plays an important role in nonlinear optics \cite{AS}, water waves \cite{BN}, and Bose--Einstein condensates \cite{WL}. Another integrable system of NLS type, the derivative-type NLS equation
\be\label{1.1}
\ii u_t+u_{xx}-\ii u^2\bar{u}_x+\frac{1}{2}|u|^4u=0,
\ee
is derived by Gerdjikov--Ivanov \cite{GI}, which is called the GI equation. Here and below, the bar refers to the complex conjugate. The GI equation can be regarded as an extension of the NLS when certain higher-order nonlinear effects are taken into account. In recent years, there has been much work on the GI equation, such as its Darboux transformation and Hamiltonian structures \cite{FAN1,FAN2}, the algebra-geometric solutions \cite{DF}, the rogue wave and breather solution \cite{XH}. Particularly, the long-time asymptotic behavior of solution to the GI equation \eqref{1.1} was established in \cite{JX2,ST} by using the nonlinear steepest descent method.

In particular, via the transformation $u(x,t)=q(x,t)\e^{\frac{1}{2}\ii q_0^4t}$, equation \eqref{1.1} is trivially converted into
\be
\ii q_t+q_{xx}-\ii q^2\bar{q}_x+\frac{1}{2}(|q|^4-q_0^4)q=0.\label{1.2}
\ee
Our purpose in the present work is devoted to the study of the long-time asymptotics of equation \eqref{1.2} on the line with symmetric, nonzero boundary conditions at infinity, that is,
\be\label{1.3}
\lim_{x\rightarrow\pm\infty}q(x,0)=q_\pm.
\ee
Hereafter, $q_\pm$ are complex constants and $|q_\pm|=q_0>0$.
This form of equation \eqref{1.2} has the advantage that the solutions of \eqref{1.2} which satisfy \eqref{1.3} are asymptotically time independent as $x\rightarrow\infty.$ The problems with nonzero boundary conditions at infinity of the kind \eqref{1.3} have already been studied. For example, the inverse scattering transform (IST) and the long-time asymptotics for the focusing NLS equation with nonzero boundary conditions at infinity were developed recently in \cite{GB2} and \cite{GB1}, respectively. Furthermore, for the multi-component case, the initial value problem for the focusing Manakov system with nonzero boundary conditions at infinity is solved in \cite{KB} by developing an appropriate IST. The three-component defocusing NLS equation with nonzero boundary conditions was analyzed in \cite{BK} by the theory of IST. On the other hand, for the asymmetric non-zero boundary conditions (i.e., when the limiting values of the solution at space infinities have different non-zero moduli), the IST for the focusing and defocusing NLS equation were formulated in \cite{FD} and \cite{GB3}, respectively.

Our present work was motivated by the long-time asymptotic analysis for the focusing NLS equation developed in \cite{GB1}. Our goal here is to compute the long-time asymptotics for the GI-type derivative NLS equation \eqref{1.2} with the initial value condition satisfied \eqref{1.3}. The main tool is the asymptotic analysis of an associated matrix Riemann--Hilbert (RH) problem by using the steepest descent method and the so-called $g$-function mechanism \cite{PD2}. The well-known nonlinear steepest descent method introduced by Deift and Zhou in \cite{PD} provides a detailed rigorous proof to calculate the large-time asymptotic behaviors of the integrable nonlinear evolution equations. This approach has been successfully applied in determining asymptotic formulas for the initial value problems of a number of integrable systems associated with $2\times2$ matrix sprectral problems including the mKdV equation \cite{PD}, the KdV equation \cite{GT}, the Hirota equation \cite{HL}, the derivative NLS equation \cite{LP}, the Fokas--Lenells equation \cite{XJ} and the Kundu--Eckhaus equation \cite{DS}. For the $3\times3$ matrix spectral problem, the large-time asymptotic behavior for the coupled NLS equations was obtained in \cite{XG} via nonlinear steepest descent. Moreover, there are also many beautiful results about the study of asymptotics of solutions of the initial value problems with shock-type oscillating initial data \cite{RB}, nondecaying step-like initial data \cite{AB3,JX2,KM} and the initial-boundary value problems with $t$-periodic boundary condition \cite{AB,ST} for various integrable equations. Recently, Lenells also has been derived some interesting asymptotic formulas for the solution of integrable equations on the half-line \cite{JL1,JL2} by using the steepest descent method. We also have done some meaningful work about determining the long-time asymptotics for integrable equations on the half-line, see \cite{GL1,GL2}.

The organization of the paper is as follows. In Section 2, we formulate the main Riemann--Hilbert problem to solve the initial value problem for the GI-type derivative NLS equation \eqref{1.2} with nonzero boundary conditions. We then use this RH problem to compute the long-time asymptotic behavior of the solution in different regions of the $xt$-plane in Section 3.

\section{The Riemann--Hilbert problem}
\setcounter{equation}{0}
\setcounter{lemma}{0}
\setcounter{theorem}{0}
In this section, we will use the approach proposed in \cite{GB1} to formulate the main RH problem, which allows
us to give a representation of the solution for the equation \eqref{1.2}.
\subsection{Eigenfunctions}
The integrability of equation \eqref{1.2} follows from its Lax pair representation
\bea
\Psi_x&=&X\Psi,\quad\Psi_t=T\Psi,\label{2.1}
\eea
where
\bea
X&=&\ii k^2\sigma_3+\ii kQ-\frac{\ii}{2}Q^2\sigma_3,\label{2.2}\\
T&=&-2\ii k^4\sigma_3-2\ii k^3Q+\ii k^2Q^2\sigma_3+kQ_x\sigma_3+\frac{\ii}{4}(Q^4-q_0^4I)\sigma_3+\frac{1}{2}(Q_xQ-QQ_x),\label{2.3}
\eea
$\Psi(x,t;k)$ is a vector or a $2\times2$ matrix-valued function and $k\in\bfC$ is the spectral parameter, and
\be
\sigma_3={\left( \begin{array}{cc}
1 ~& 0 \\
0 ~& -1 \\
\end{array}
\right )},\quad Q=Q(x,t)={\left( \begin{array}{cc}
0 ~& q(x,t) \\
-\bar{q}(x,t) ~& 0 \\
\end{array}
\right )}.
\ee
Let $X_\pm=\lim_{x\rightarrow\pm\infty}X(x,t;k)$ and $T_\pm=\lim_{x\rightarrow\pm\infty}T(x,t;k)$. It is straightforward to see that the eigenvector matrix of $X_\pm$ can be written as
\be\label{2.5}
E_\pm=\begin{pmatrix}
1 ~& \frac{\lambda-(k^2+\frac{1}{2}q_0^2)}{k\bar{q}_\pm}\\[4pt]
\frac{\lambda-(k^2+\frac{1}{2}q_0^2)}{kq_\pm} ~& 1
\end{pmatrix},
\ee
while the corresponding eigenvalues $\pm\ii\lambda$ are defined by
\be
\lambda(k)=\bigg(k^4+\frac{1}{4}q_0^4\bigg)^{\frac{1}{2}}.\label{2.6}
\ee
The branch cut for $\lambda(k)$ is taken along the segment
\be
\varsigma\cup\bar{\varsigma}=\{k\in\bfC|k_1^2-k_2^2=0,~|k_1k_2|\leq\frac{1}{4}q_0^2\},
\ee
where $\varsigma=\{k\in\bfC|k_1^2-k_2^2=0,~|k_1k_2|\leq\frac{1}{4}q_0^2,~\text{Im}k^2>0\},~k_1=\text{Re}k,~k_2=\text{Im}k.$

On the other hand, one obtains that $T_\pm=-2k^2X_\pm$, we seek simultaneous solution $\Psi_\pm$ of Lax pair \eqref{2.1} such that
\be\label{2.8}
\Psi_\pm(x,t;k)=E_\pm(k)\e^{\ii\Phi(x,t;k)\sigma_3}(1+o(1)),\quad x\rightarrow\pm\infty,
\ee
where \be
\Phi(x,t;k)=\lambda(x-2k^2t).
\ee
We also find that for any $t\geq0$, $\Psi_\pm(x,t;k)$ remain bounded as $x\rightarrow\pm\infty$ if and only if $k\in\Sigma_0=\bfR\cup\ii\bfR\cup\varsigma\cup\bar{\varsigma}$. The contour $\Sigma_0$ can be reduced to the contour $\Sigma=\bfR\cup\gamma\cup\bar{\gamma}$ in the $k^2$-plane as shown in Fig. \ref{fig1}. Introducing a new eigenfunction $\mu(x,t;k)$ by
\begin{equation}\label{2.10}
\mu_\pm(x,t;k)=\Psi_\pm(x,t;k) \e^{-\ii\Phi(x,t;k)\sigma_3},
\end{equation}
we obtain that
\begin{equation}\label{2.11}
\mu_\pm(x,t;k)=E_\pm(k)+o(1),\quad x\rightarrow\pm\infty,~k\in\Sigma_0.
\end{equation}
\begin{figure}[htbp]
  \centering
  \includegraphics[width=3in]{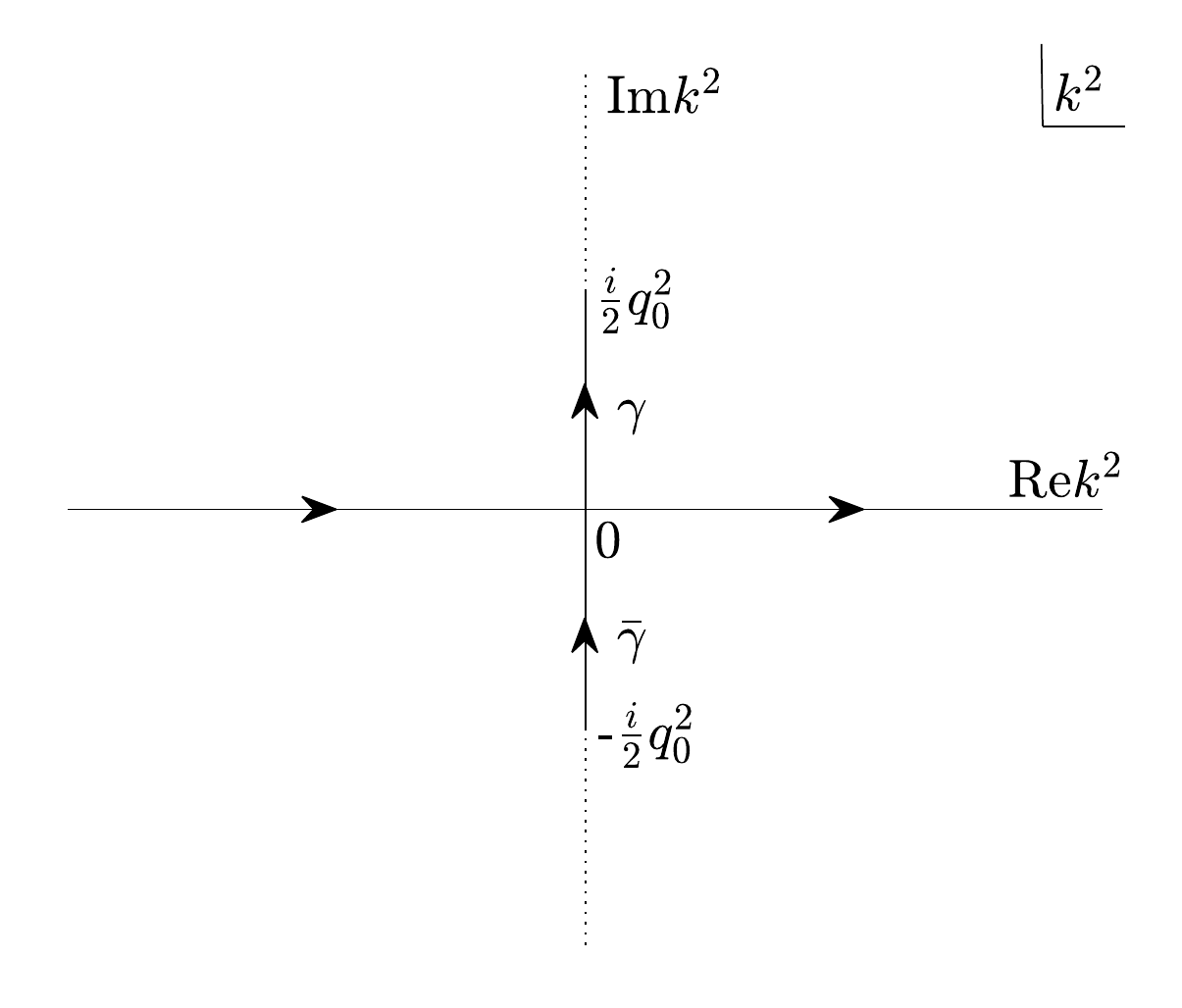}
  \caption{The oriented contour $\Sigma=\bfR\cup\gamma\cup\bar{\gamma}$ in $k^2$-plane.}\label{fig1}
\end{figure}

Using the method of variation of parameters, we get the following Volterra integral equations for $\mu_\pm$
\bea
\mu_+(x,t;k)&=&E_+(k)+\int_{+\infty}^xE_+(k)\e^{\ii\lambda(x-y)\sigma_3}E^{-1}_+(k)(\triangle X_+\mu_+)(y,t;k)\e^{-\ii\lambda(x-y)\sigma_3}\dd y,\label{2.12}\\
\mu_-(x,t;k)&=&E_-(k)+\int_{-\infty}^xE_-(k)\e^{\ii\lambda(x-y)\sigma_3}E^{-1}_-(k)(\triangle X_-\mu_-)(y,t;k)\e^{-\ii\lambda(x-y)\sigma_3}\dd y,\label{2.13}
\eea
where $\triangle X_\pm(x,t;k)=X(x,t;k)-X_\pm=\ii k(Q-Q_\pm)-\frac{\ii}{2}(Q^2+q_0^2I)\sigma_3$ and $Q_\pm=\lim_{x\rightarrow\pm\infty}Q(x,t)$.
Assuming $(q(x,t)-q_\pm)\in L^{1,1}(\bfR^\pm_x)$, and $(|q(x,t)|^2-q_0^2)\in L^{1,1}(\bfR^\pm_x)$,
where
\berr
L^{1,1}(\bfR^\pm)=\bigg\{f:\bfR\rightarrow\bfC\bigg|\int_{\bfR^\pm}(1+|x|)|f(x)|\dd x<\infty\bigg\},
\eerr
the analysis of the Neumann series for the integral equations \eqref{2.12} and \eqref{2.13} allows one to prove existence and uniqueness of the eigenfunctions $\mu_\pm$ for all $k\in\Sigma_0$ (the detailed proof can be founded \cite{GB3,GB2}). We denote by $\mu_{\pm}^{(1)}(x,t;k)$ and $\mu_{\pm}^{(2)}(x,t;k)$ the columns of $\mu_\pm(x,t;k)$. Then we have the following lemma.
\begin{lemma}\label{lem2.1}
For all $(x,t)$, the matrices $\mu_\pm(x,t;k)$ have the following properties:

(i) The determinants of $\mu_\pm(x,t;k)$ satisfy
\be\label{2.14}
\det\mu_\pm(x,t;k)=\det E_\pm(k)=\frac{2\lambda}{\lambda+k^2+\frac{1}{2}q_0^2}\triangleq d(k).
\ee

(ii) $\mu^{(1)}_+$ and $\mu^{(2)}_-$ are analytic in $\{k\in\bfC|\mbox{Im}k^2>0\}\setminus\gamma$, and $(\mu^{(1)}_+~\mu^{(2)}_-)\rightarrow I$ as $k\rightarrow\infty.$

(iii) $\mu^{(1)}_-$ and $\mu^{(2)}_+$ are analytic in $\{k\in\bfC|\mbox{Im}k^2<0\}\setminus\bar{\gamma}$, and $(\mu^{(1)}_-~\mu^{(2)}_+)\rightarrow I$ as $k\rightarrow\infty.$

(iv)  Symmetry:
\bea
\overline{\mu_\pm(x,t;\bar{k})}&=&\begin{pmatrix}
0 & 1\\
1 & 0\\
\end{pmatrix}\mu_\pm(x,t;k)\begin{pmatrix}
0 & 1\\
1 & 0\\
\end{pmatrix},\label{2.15}\\
\mu_\pm(x,t;-k)&=&\sigma_3\mu_\pm(x,t;k)\sigma_3.\label{2.16}
\eea

(v) Moreover,
\berr
\mu_\pm(x,t;k)=I+\frac{\tilde{\mu}(x,t)}{k}+O(k^{-1}),\quad k\rightarrow\infty,
\eerr
 where
\berr
[\sigma_3,\tilde{\mu}(x,t)]=\begin{pmatrix}
0 ~& -q(x,t)\\[4pt]
\bar{q}(x,t) ~& 0
\end{pmatrix}.
\eerr
\end{lemma}
\begin{remark}
The definition \eqref{2.6} of $\lambda$ implies that $d$ is nonzero and nonsingular for all $k^2\in\Sigma^*=\Sigma\setminus\{ \pm\frac{\ii}{2}q_0^2\}$. The symmetry relations \eqref{2.15} and \eqref{2.16} can be proved easily due to the symmetries of the matrix $X(x,t;k)$:
$$\overline{X(x,t;\bar{k})}=\begin{pmatrix}
0 & 1\\
1 & 0\\
\end{pmatrix}X(x,t;k)\begin{pmatrix}
0 & 1\\
1 & 0\\
\end{pmatrix},~X(x,t;-k)=\sigma_3X(x,t;k)\sigma_3.$$
The large $k$ asymptotics of $\mu_\pm(x,t;k)$ can be obtained from the $x$-part of Lax pair \eqref{2.1} and \eqref{2.10} as well as the asymptotics of $\lambda(k)$, that is, $\lambda(k)=k^2+O(k^{-2})$ as $k\rightarrow\infty$.
\end{remark}

Moreover, any two solutions of \eqref{2.1} are related, thus we define the spectral matrix $s(k)$ by
\be\label{2.17}
\mu_-(x,t;k)=\mu_+(x,t;k)\e^{\ii\Phi(x,t;k)\sigma_3}s(k)\e^{-\ii\Phi(x,t;k)\sigma_3},\quad k^2\in\Sigma^*.
\ee
In particular, it follows from \eqref{2.14} that
\be\label{2.18}
\det s(k)=1.
\ee
Due to the symmetry relation \eqref{2.15}, the matrix-valued spectral functions $s(k)$ can be  defined in terms of two scalar spectral functions, $a(k)$ and $b(k)$ by
\be\label{2.19}
s(k)=\begin{pmatrix}
\bar{a}(k) & b(k)\\[4pt]
\bar{b}(k) & a(k)\\
\end{pmatrix},\quad a(k)\bar{a}(k)-b(k)\bar{b}(k)=1,
\ee
where $\bar{a}(k)=\overline{a(\bar{k})}$ and $\bar{b}(k)=\overline{b(\bar{k})}$ indicate the Schwartz conjugates.
Meanwhile, by the symmetry relation \eqref{2.16}, one can infer that
\be\label{2.20}
a(k)=a(-k),\quad b(k)=-b(-k).
\ee
Finally, equations \eqref{2.10}, \eqref{2.17} and \eqref{2.19} imply
\bea
a(k)&=&\frac{1}{d(k)}\text{wr}(\Psi_+^{(1)}(x,t;k),\Psi_-^{(2)}(x,t;k)),\label{2.21}\\
b(k)&=&\frac{1}{d(k)}\text{wr}(\Psi_-^{(2)}(x,t;k),\Psi_+^{(2)}(x,t;k)),\label{2.22}\\
\bar{a}(k)&=&\frac{1}{d(k)}\text{wr}(\Psi_-^{(1)}(x,t;k),\Psi_+^{(2)}(x,t;k)),\label{2.23}\\
\bar{b}(k)&=&\frac{1}{d(k)}\text{wr}(\Psi_+^{(1)}(x,t;k),\Psi_-^{(1)}(x,t;k)),\label{2.24}
\eea
where wr denotes the Wronskian determinant. Thus, we can find that $a(k)$ is analytic in $\{k\in\bfC|\mbox{Im}k^2>0\}\setminus\gamma$.

The jump discontinuity of $\lambda$ across $\gamma\cup\bar{\gamma}$ induces a corresponding jump for the eigenfunctions and scattering data.
\begin{lemma}\label{lem2.2}
The columns of fundamental solutions $\Psi(x,t;k)$ and scattering data $a(k)$, $b(k)$ satisfy the following
jump conditions across $\gamma\cup\bar{\gamma}$:
\bea
\begin{aligned}
(\Psi_+^{(1)})^+(x,t;k)&=-\frac{\lambda+k^2+\frac{1}{2}q_0^2}{kq_+}\Psi_+^{(2)}(x,t;k),\\
(\Psi_-^{(2)})^+(x,t;k)&=-\frac{\lambda+k^2+\frac{1}{2}q_0^2}{k\bar{q}_-}\Psi_-^{(1)}(x,t;k),
\end{aligned}\quad k^2\in\gamma,\label{2.25}\\
\begin{aligned}
(\Psi_-^{(1)})^+(x,t;k)&=-\frac{\lambda+k^2+\frac{1}{2}q_0^2}{kq_-}\Psi_-^{(2)}(x,t;k),\\
(\Psi_+^{(2)})^+(x,t;k)&=-\frac{\lambda+k^2+\frac{1}{2}q_0^2}{k\bar{q}_+}\Psi_+^{(1)}(x,t;k),
\end{aligned}\quad k^2\in\bar{\gamma},
\eea
and
\be\label{2.27}
a^+(k)=\frac{q_-}{q_+}\bar{a}(k),\quad k^2\in\gamma.
\ee
\end{lemma}
\begin{proof}
It is noted that the pairs $\{\Psi_-^{(1)},\Psi_-^{(2)}\}$ and $\{\Psi_+^{(1)},\Psi_+^{(2)}\}$ are both fundamental sets of solutions of $x$-part Lax pair \eqref{2.1}. Thus, the limit $(\Psi_+^{(1)})^+(x,t;k)$ satisfies
\berr
(\Psi_+^{(1)})^+(x,t;k)=\beta_1\Psi_+^{(1)}(x,t;k)+\beta_2\Psi_+^{(2)}(x,t;k),\quad k^2\in\gamma,
\eerr
where $\beta_1,~\beta_2$ are independent of $x$. Letting $x\rightarrow+\infty$ on both side of above equation, we get from \eqref{2.5} and \eqref{2.8} that
\berr
\beta_1=0,\quad \beta_2=-\frac{\lambda+k^2+\frac{1}{2}q_0^2}{kq_+},
\eerr
which yields the first relation of \eqref{2.25}. The others follow the similar arguments. Combining \eqref{2.21} with \eqref{2.25}, \eqref{2.23}, we can easily obtain \eqref{2.27}.
\end{proof}
\subsection{The Riemann--Hilbert problem and the solution of the Cauchy problem}
The scattering relation \eqref{2.17} can be rewritten in the form of conjugation of boundary values of a piecewise analytic matrix-valued function on a contour in the complex $k$-plane. The final form is
\be\label{2.28}
M^+(x,t;k)=M^-(x,t;k)J(x,t;k),~~k^2\in\Sigma,
\ee
where $M^{\pm}(x,t;k)$ denote the boundary values of $M(x,t;k)$ according to a chosen orientation of $\Sigma$, see Fig. \ref{fig1}. Indeed, define the matrix $M(x,t;k)$ as follows:
\be\label{2.29}
M(x,t;k)=\left\{
\begin{aligned}
&\bigg(\frac{\mu_+^{(1)}}{ad},\mu_-^{(2)}\bigg)
=\bigg(\frac{\Psi_+^{(1)}}{ad},\Psi_-^{(2)}\bigg)\e^{-\ii\Phi\sigma_3},\quad \text{Im}k^2>0\setminus\gamma,\\
&\bigg(\mu_-^{(1)}, \frac{\mu_+^{(2)}}{\bar{a}d}\bigg)=\bigg(\Psi_-^{(1)}, \frac{\Psi_+^{(2)}}{\bar{a}d}\bigg)\e^{-\ii\Phi\sigma_3},~\quad \text{Im}k^2<0\setminus\bar{\gamma},\\
\end{aligned}
\right.
\ee
where the dependence on the variables $x$, $t$, $k$ on the right-hand side has been suppressed for brevity. Then the boundary values $M^{\pm}(x,t;k)$ relative to $\Sigma$ are related by \eqref{2.28} with the jump matrix $J$ is given by
\bea\label{2.30}
J(x,t;k)=\left\{
\begin{aligned}
&\begin{pmatrix}
\frac{1}{d(k)}[1-r(k)\bar{r}(k)] ~& -\bar{r}(k)\e^{2\ii\Phi(x,t;k)}\\[4pt]
r(k)\e^{-2\ii\Phi(x,t;k)}~ & d(k)
\end{pmatrix},\qquad\qquad\qquad\quad~~k^2\in\bfR,\\
&\begin{pmatrix}
-\frac{\lambda-(k^2+\frac{1}{2}q_0^2)}{kq_-}\bar{r}(k)\e^{2\ii\Phi(x,t;k)}~& -\frac{2\lambda}{k\bar{q}_-}\\[4pt]
\frac{k\bar{q}_-}{2\lambda}[1-r(k)\bar{r}(k)] ~& \frac{\lambda+k^2+\frac{1}{2}q_0^2}{k\bar{q}_-}r(k)\e^{-2\ii\Phi(x,t;k)}
\end{pmatrix},\quad k^2\in\gamma,\\
&\begin{pmatrix}
\frac{\lambda+k^2+\frac{1}{2}q_0^2}{kq_-}\bar{r}(k)\e^{2\ii\Phi(x,t;k)} ~& \frac{kq_-}{2\lambda}[1-r(k)\bar{r}(k)]\\[4pt]
-\frac{2\lambda}{kq_-} ~& -\frac{\lambda-(k^2+\frac{1}{2}q_0^2)}{k\bar{q}_-}r(k)\e^{-2\ii\Phi(x,t;k)}
\end{pmatrix},\quad k^2\in\bar{\gamma},
\end{aligned}
\right.
\eea
where \be
r(k)=-\frac{\bar{b}(k)}{a(k)}.
\ee
We note that the jump of $M$ across $\gamma\cup\bar{\gamma}$ is obtained according to the jump conditions in Lemma \ref{lem2.2}. Finally, it follows from Lemma \ref{lem2.1} and the relationship \eqref{2.29} between $M$ and $\mu_\pm$, it can be shown that $M$ admits the following large $k$ asymptotic expansion
\be
M(x,t;k)=I+\frac{M_1(x,t)}{k}+O(k^{-2}),\quad k\rightarrow\infty.
\ee
Thus, the solution $q(x,t)$ of the GI-type derivative NLS equation \eqref{1.2} can be expressed in terms of the solution of the basic RH problem as follows:
\be\label{2.33}
q(x,t)=-2(M_1(x,t))_{12}=-2\lim_{k\rightarrow\infty}(kM(x,t;k))_{12},
\ee
where $M$ is the solution of the following RH problem:\\
Suppose that $a(k)\neq0$ for $\{\text{Im}k^2>0\}\cup\Sigma$. Determine a $2\times2$ matrix-valued function $M(x,t;k)$ such that

$\bullet$ $M(x,t;k)$ is a sectionally meromorphic function in $\bfC\setminus\{k^2\in\Sigma\}$;

$\bullet$ $M(x,t;k)$ satisfies the jump condition in \eqref{2.28} with the jump matrix given by \eqref{2.30};

$\bullet$  $M(x,t;k)$ has the following asymptotics:
\be
M(x,t;k)=I+O(k^{-1}),~k\rightarrow\infty.
\ee

Although it is possible to perform this analysis directly in the complex $k$-plane, the symmetry of the spectral functions $a(k),~b(k)$ in \eqref{2.20} suggest that it is convenient to introduce a new spectral variable $z$ by $z=k^2$. This change of spectral parameter appeared already in \cite{KN} and was further employed in \cite{JL2,JX2}. One advantage of working with $z$ is that we can more easily analyze the signature table of the real part for the new phase function in order to find the long-time asymptotics.

The symmetry relation \eqref{2.16} implies that the solution $M(x,t;k)$ obeys the symmetry
\be\label{2.35}
M(x,t;k)=\sigma_3 M(x,t;-k)\sigma_3,\quad k\in\bfC\setminus\Sigma_0.
\ee
Hence, letting $z=k^2$, we can define $m(x,t;z)$ by the equation
\be\label{2.36}
m(x,t;z)=\begin{pmatrix}
1\quad & 0\\[4pt]
\frac{1}{2}\bar{q} \quad& 1\\
\end{pmatrix}k^{-\frac{\hat{\sigma}_3}{2}}M(x,t;k),\quad z\in\bfC\setminus\Sigma.
\ee
The factor $k^{-\frac{\hat{\sigma}_3}{2}}$ is included in \eqref{2.36} in order to make the right-hand side an even function of $k$ and $k^{-\frac{\hat{\sigma}_3}{2}}$ acts on a $2\times2$ matrix $A$ by $k^{-\frac{\sigma_3}{2}}Ak^{\frac{\sigma_3}{2}}$; the matrix involving $\bar{q}$ is included to ensure that $m\rightarrow I$ as $z\rightarrow\infty$. We define the new spectral function $\rho(z)$ by
\be\label{2.37}
\rho(z)=\frac{r(k)}{k},\quad z\in\bfR.
\ee
Then the Riemann--Hilbert problem for $M(x,t;k)$ can be rewritten in terms of $m(x,t;z)$ as
\be\label{2.38}
\begin{aligned}
&m^+(x,t;z)=m^-(x,t;z)J^{(0)}_1(x,t;z),\quad z\in\bfR,\\
&m^+(x,t;z)=m^-(x,t;z)J^{(0)}_2(x,t;z),\quad z\in\gamma,\\
&m^+(x,t;z)=m^-(x,t;z)J^{(0)}_3(x,t;z),\quad z\in\bar{\gamma},\\
\end{aligned}
\ee
with the normalization condition
\be\label{2.39}
m(x,t;z)=I+O(z^{-1}),\quad z\rightarrow\infty,
\ee
where
\bea
\begin{aligned}
J^{(0)}_1(x,t;z)&=\begin{pmatrix}
\frac{1}{d(z)}[1-z\rho(z)\bar{\rho}(z)] ~ & -\bar{\rho}(z)\e^{2\ii t\theta(\xi;z)}\\[4pt]
z\rho(z)\e^{-2\ii t\theta(\xi;z)} ~& d(z)\\
\end{pmatrix},\\
J^{(0)}_2(x,t;z)&=\begin{pmatrix}
-\frac{\lambda-(z+\frac{1}{2}q_0^2)}{q_-}\bar{\rho}(z)\e^{2\ii t\theta(\xi;z)}~& -\frac{2\lambda}{z\bar{q}_-}\\[4pt]
\frac{z\bar{q}_-}{2\lambda}[1-z\rho(z)\bar{\rho}(z)] ~& \frac{\lambda+z+\frac{1}{2}q_0^2}{\bar{q}_-}\rho(z)\e^{-2\ii t\theta(\xi;z)}
\end{pmatrix},\\
J^{(0)}_3(x,t;z)&=\begin{pmatrix}
\frac{\lambda+z+\frac{1}{2}q_0^2}{q_-}\bar{\rho}(z)\e^{2\ii t\theta(\xi;z)} ~& \frac{q_-}{2\lambda}[1-z\rho(z)\bar{\rho}(z)]\\[4pt]
-\frac{2\lambda}{q_-} ~& -\frac{\lambda-(z+\frac{1}{2}q_0^2)}{\bar{q}_-}\rho(z)\e^{-2\ii t\theta(\xi;z)}
\end{pmatrix},
\end{aligned}
\eea
and
\bea
\lambda(z)&=&\bigg(z^2+\frac{1}{4}q_0^4\bigg)^{\frac{1}{2}}, \quad d(z)=\frac{2\lambda}{\lambda+z+\frac{1}{2}q_0^2},\\
\theta(\xi;z)&=&\frac{\Phi(x,t;k)}{t}=\lambda(\xi-2z),\quad\xi=\frac{x}{t}.
\eea
In terms of $m(x,t;z)$, \eqref{2.33} can be expressed as
\bea\label{2.43}
q(x,t)=-2\lim_{z\rightarrow\infty}(z m(x,t;z))_{12}.
\eea

\section{The long-time asymptotics}
\setcounter{equation}{0}
\setcounter{lemma}{0}
\setcounter{theorem}{0}
\setcounter{remark}{0}
In this section, we compute the long-time asymptotic behavior of the solution $q(x,t)$ of the GI-type derivative NLS equation \eqref{1.2}, as obtained by equation \eqref{2.43} by using the Deift-Zhou nonlinear steepest descent method to perform the asymptotic analysis of oscillating RH problems \cite{GB1,PD,AB,AB3,RB}. The key point of this method is the choice of contour deformations, which depends crucially on the sign structure of the quantity Re$(\ii\theta)$. We note that
\be
\theta(\xi;z)=\sqrt{z^2+\frac{1}{4}q_0^4}(\xi-2z).
\ee
Thus, we get
\be
\frac{\dd\theta(\xi;z)}{\dd z}=\frac{-4z^2+\xi z-\frac{1}{2}q_0^4}{\lambda},
\ee
which implies that $\theta$ has two stationary points
\be\label{3.3}
z_{\pm}=\frac{1}{8}(\xi\pm\sqrt{\xi^2-8q_0^2}).
\ee
Letting $z=z_1+\ii z_2$, we have
\berr
\theta^2(\xi;z)=(z_1^2-z_2^2+\frac{1}{4}q_0^4+2\ii z_1z_2)(4z_1^2-4z_2^2+\xi^2-4\xi z_1+\ii(8z_1z_2-4\xi z_2)).
\eerr
Hence, according to
\berr
\{\text{Im}\theta(\xi;z)=0\}=\{\text{Im}\theta^2(\xi;z)=0\}\cap\{\text{Re}\theta^2(\xi;z)\geq0\},
\eerr
we can find the curves of the sets $\{\text{Im}\theta(\xi;z)=0\}$. In fact, the sign structure of Re(i$\theta)$ in the complex $z$-plane is shown in Fig. \ref{fig2}.
\begin{remark}
It is noted that $|\xi|>2\sqrt{2}q_0^2$, that is, $x<-2\sqrt{2}q_0^2t$ or $x>2\sqrt{2}q_0^2t$, the function $\theta$ has two real stationary points. We will show that this sectors of the $xt$-plane correspond to plane wave regions, whereas the sectors where $|\xi|<2\sqrt{2}q_0^2$ correspond to modulated elliptic wave regions.
\end{remark}
\begin{figure}[htbp]
\begin{minipage}[t]{0.3\linewidth}
\includegraphics[width=2.1in]{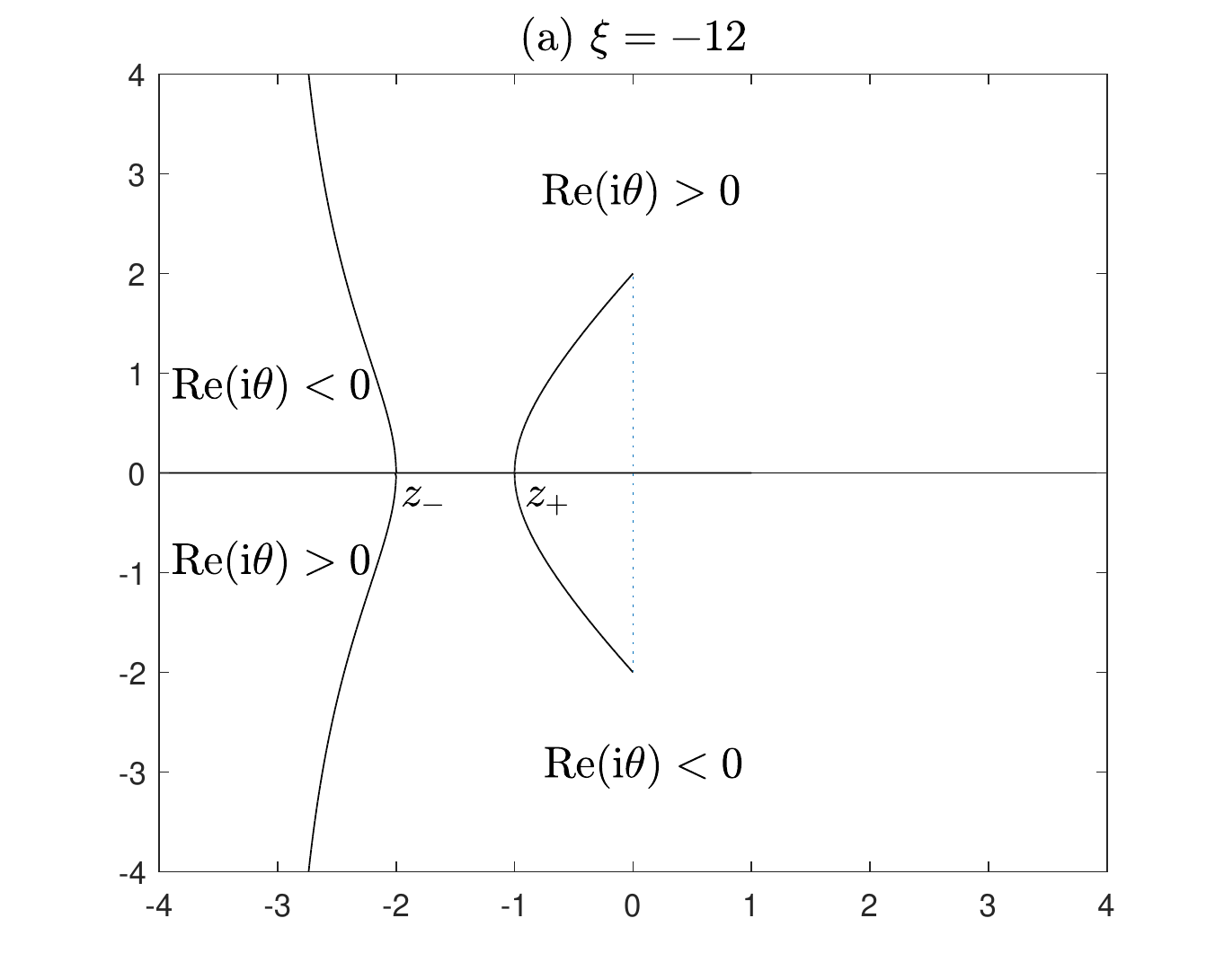}
\label{fig:side:a}
\end{minipage}%
\begin{minipage}[t]{0.3\linewidth}
\includegraphics[width=2.1in]{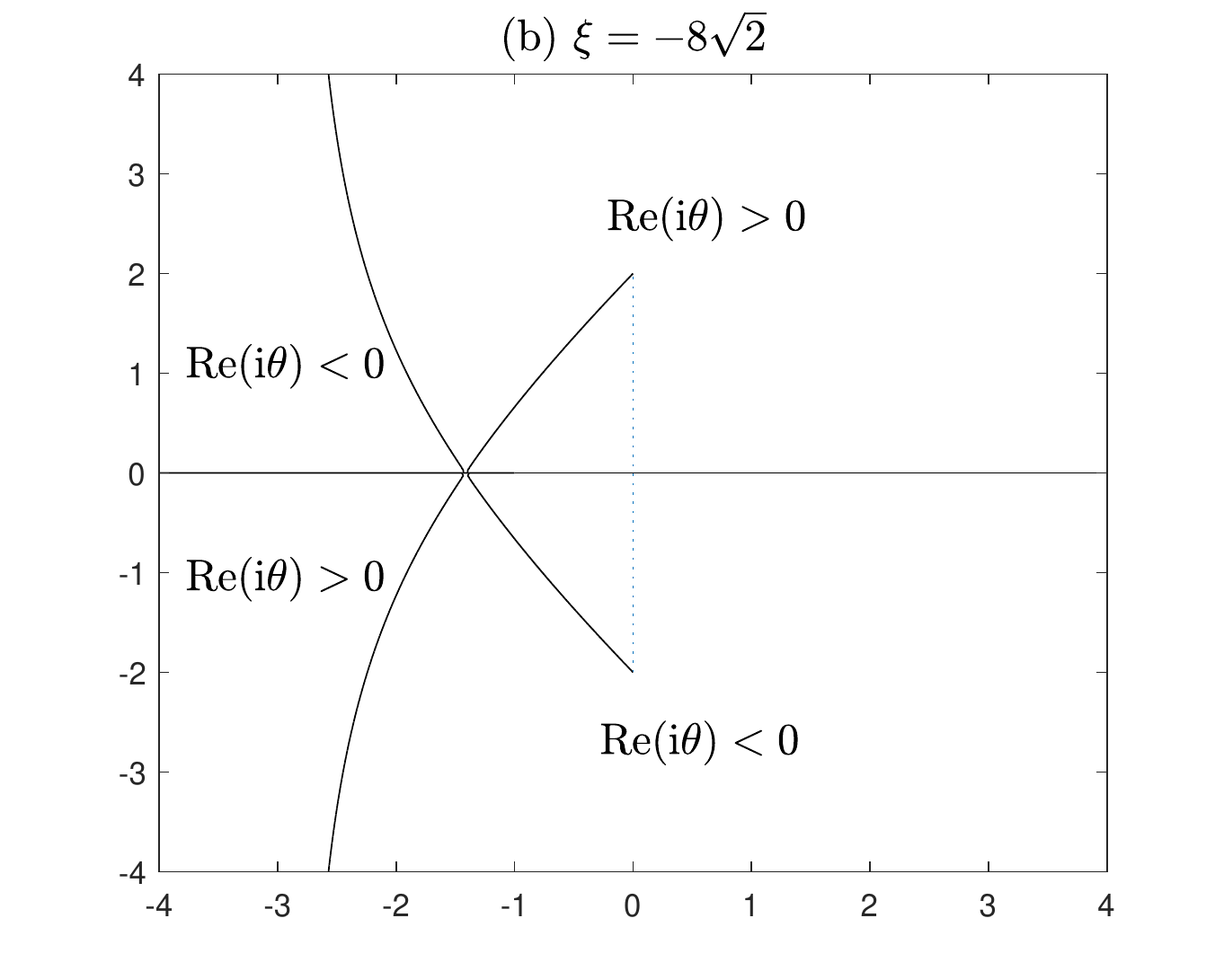}
\label{fig:side:b}
\end{minipage}
\begin{minipage}[t]{0.3\linewidth}
\includegraphics[width=2.1in]{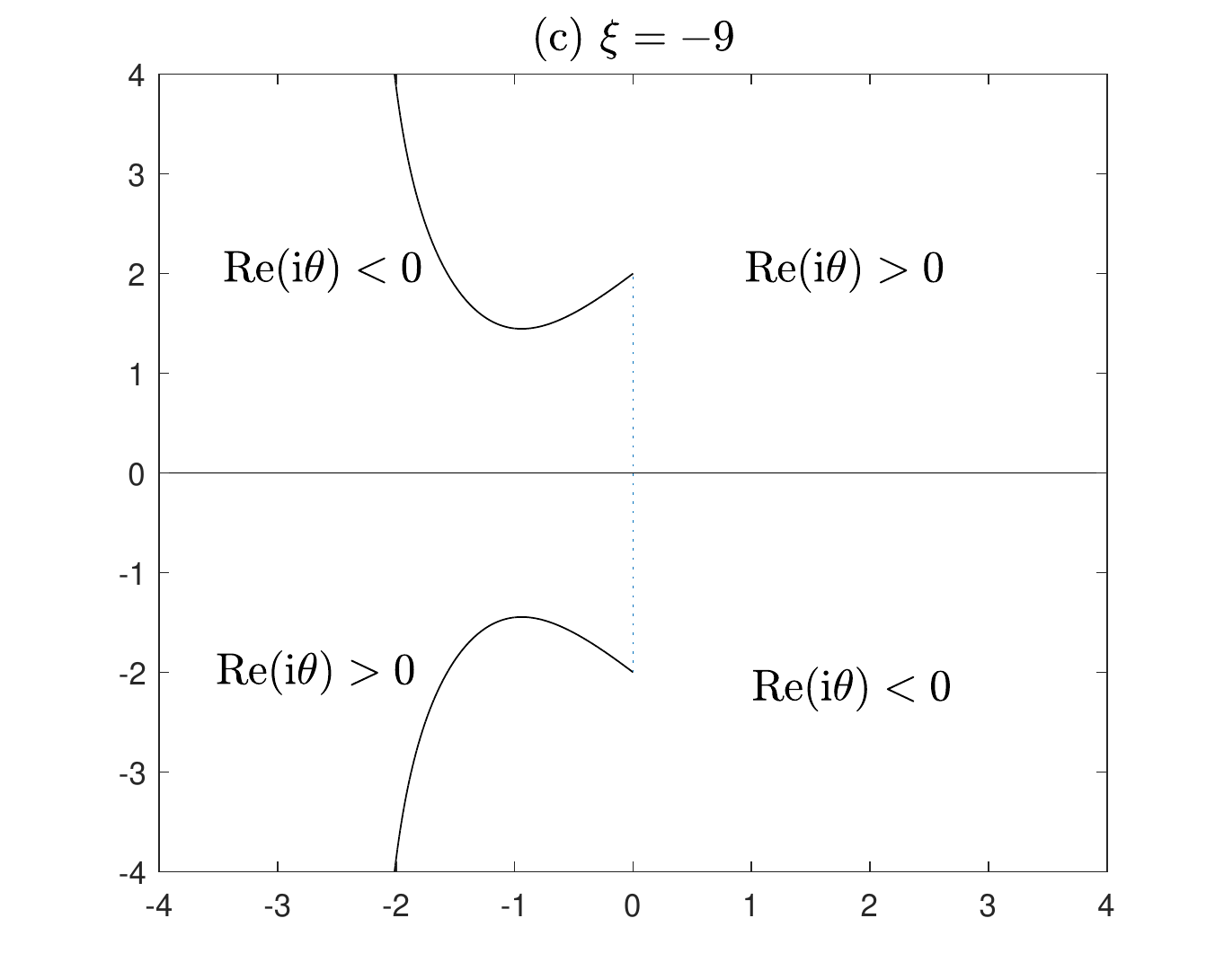}
\label{fig:side:b}
\end{minipage}\\
\begin{minipage}[t]{0.3\linewidth}
\includegraphics[width=2.1in]{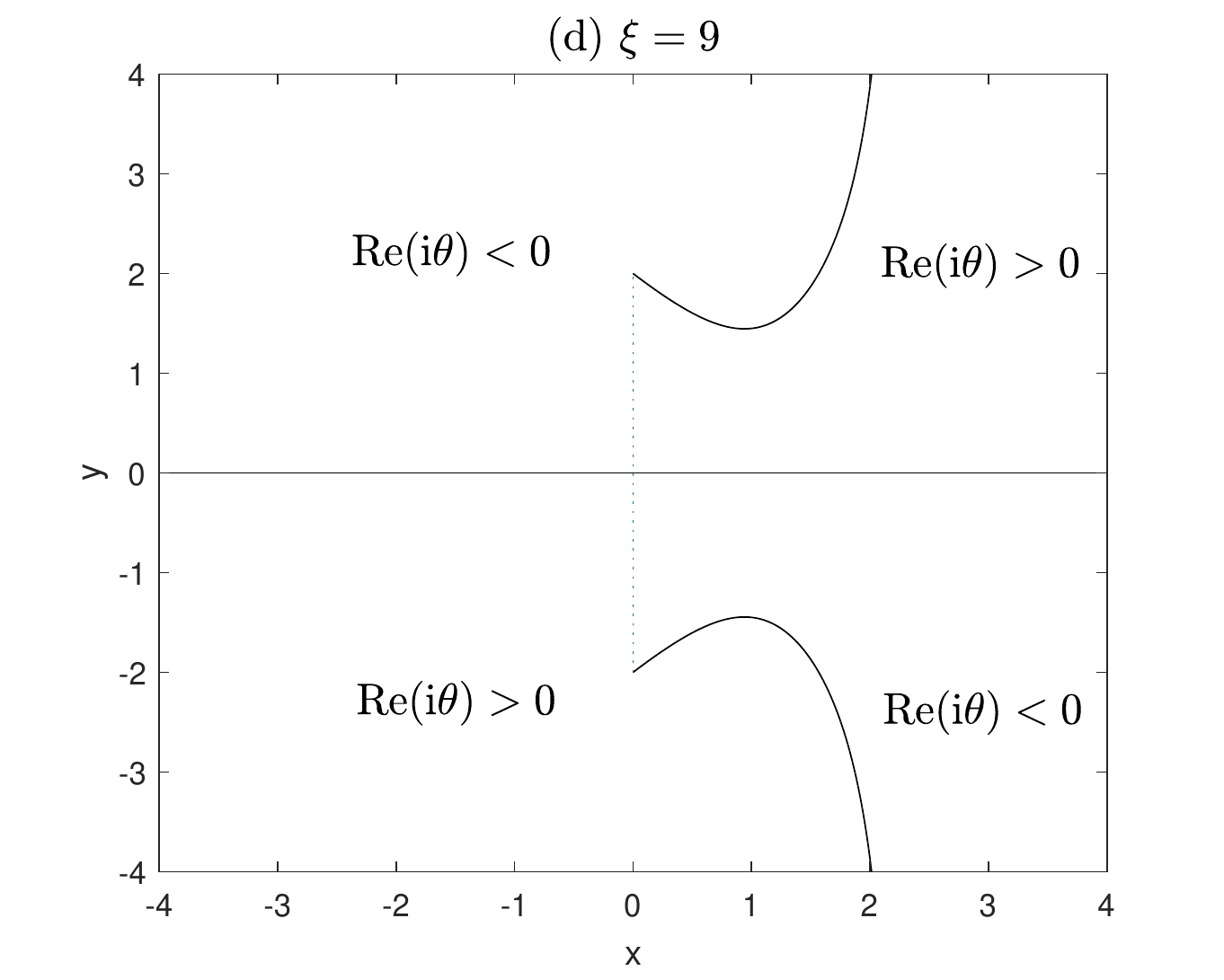}
\label{fig:side:a}
\end{minipage}%
\begin{minipage}[t]{0.3\linewidth}
\includegraphics[width=2.1in]{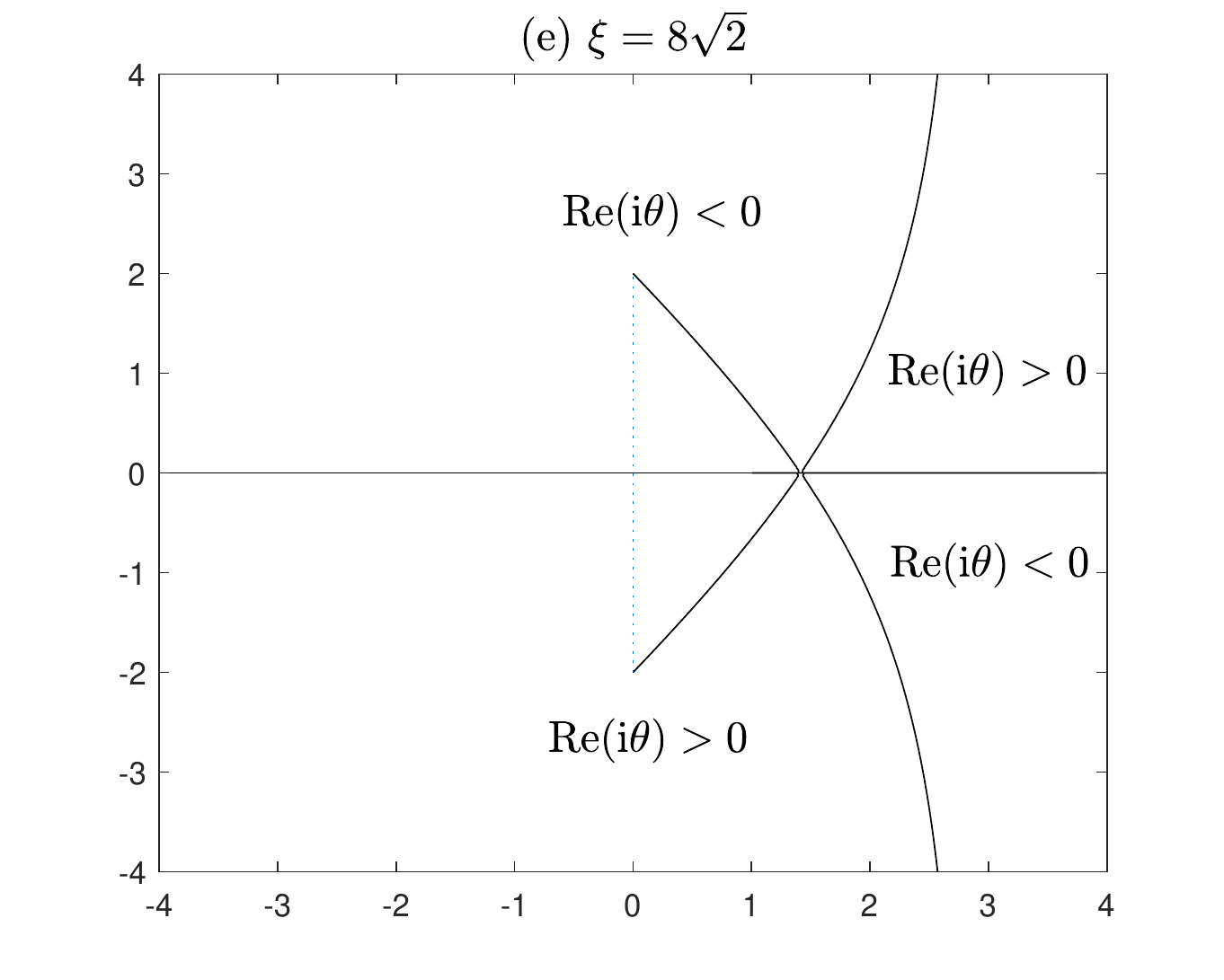}
\label{fig:side:b}
\end{minipage}
\begin{minipage}[t]{0.3\linewidth}
\includegraphics[width=2.1in]{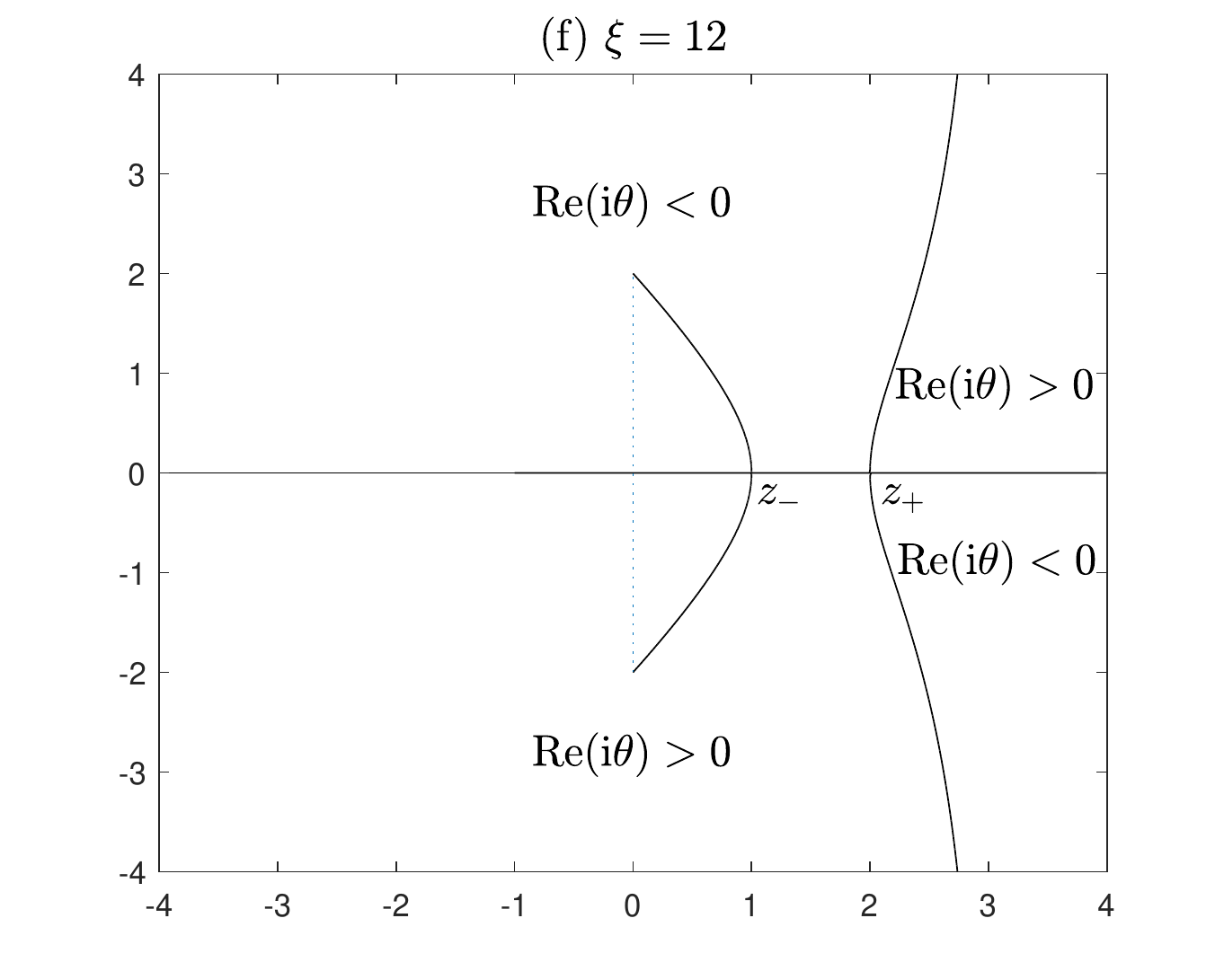}
\label{fig:side:b}
\end{minipage}
\caption{The signature table for Re(i$\theta)$ in the complex $z$-plane for various values of $\xi$ and $q_0=2$.}\label{fig2}
\end{figure}

\subsection{The plane wave region I: $x<-2\sqrt{2}q_0^2t$}\label{sec3.1}

For $x<-2\sqrt{2}q_0^2t$, that is, $\xi<-2\sqrt{2}q_0^2$, $\theta(\xi;z)$ has two real stationary points $z_\pm$ given by \eqref{3.3}. This implies that one can introduce a complex-valued function $\delta(z)$ to transform the original RH problem for $m(x,t;z)$ to the RH problem for the new function $m^{(1)}(x,t;z)$ by
\be
m^{(1)}(x,t;z)=m(x,t;z)\delta^{-\sigma_3}(z),
\ee
where
\be
\delta(z)=\exp\bigg\{\frac{1}{2\pi\ii}\int_{-\infty}^{z_-}\frac{\ln[1-s\rho(s)\bar{\rho}(s)]}{s-z}\dd s\bigg\},\quad z\in\bfC\setminus(-\infty,z_-].
\ee
\begin{lemma}
The function $\delta(z)$ has the following properties:

(i) $\delta(z)$ satisfies the following jump condition across the real axis oriented in Fig. \ref{fig1}:
\berr
\delta^+(z)=
&\delta^-(z)[1-z\rho(z)\bar{\rho}(z)],~~z\in(-\infty,z_-).
\eerr

(ii) As $z\rightarrow\infty$, $\delta(z)$ satisfies the asymptotic formula
\be\label{3.7}
\delta(z)=1+O(z^{-1}),\quad z\rightarrow\infty.
\ee

(iii) $\delta(z)$ and $\delta^{-1}(z)$ are bounded and analytic functions of $z\in\bfC\setminus(-\infty,z_-]$ with continuous boundary values on $(-\infty,z_-)$.

(iv) $\delta(z)$ obeys the symmetry $$\delta(z)=\overline{\delta(\bar{z})}^{-1},\quad z\in\bfC\setminus(-\infty,z_-].$$
\end{lemma}
Then $m^{(1)}(x,t;z)$ satisfies the jump condition
\bea
\begin{aligned}
&m^{(1)+}(x,t;z)=m^{(1)-}(x,t;z)J^{(1)}(x,t;z),\\
&z\in\Sigma^{(1)}=\Sigma=\bfR\cup\gamma\cup\bar{\gamma},
\end{aligned}
\eea
where the jump matrix $J^{(1)}=(\delta^-)^{\sigma_3}J^{(0)}(\delta^+)^{-\sigma_3}$ is given by
\bea
\begin{aligned}
J^{(1)}_1&=\begin{pmatrix}
d^{-\frac{1}{2}} ~ & 0\\[4pt]
\frac{z\rho\e^{-2\ii t\theta}(\delta^-)^{-2}}{1-z\rho\bar{\rho}}d^{\frac{1}{2}} ~& d^{\frac{1}{2}}\\
\end{pmatrix}\begin{pmatrix}
d^{-\frac{1}{2}} ~ & -\frac{\bar{\rho}\e^{2\ii t\theta}(\delta^+)^2}{1-z\rho\bar{\rho}}d^{\frac{1}{2}}\\[4pt]
0 ~& d^{\frac{1}{2}}\\
\end{pmatrix},\quad z\in(-\infty,z_-),\\
J^{(1)}_2&=\begin{pmatrix}
d^{-\frac{1}{2}} ~ & -\bar{\rho}\e^{2\ii t\theta}\delta^2d^{-\frac{1}{2}}\\[4pt]
0 ~& d^{\frac{1}{2}}\\
\end{pmatrix}\begin{pmatrix}
d^{-\frac{1}{2}} ~ & 0\\[4pt]
z\rho\e^{-2\ii t\theta}\delta^{-2}d^{-\frac{1}{2}} ~& d^{\frac{1}{2}}\\
\end{pmatrix},\quad z\in(z_-,\infty),\\
J^{(1)}_3&=\begin{pmatrix}
-\frac{\lambda-(z+\frac{1}{2}q_0^2)}{q_-}\bar{\rho}\e^{2\ii t\theta}~& -\frac{2\lambda\delta^2}{z\bar{q}_-}\\[4pt]
\frac{z\bar{q}_-\delta^{-2}}{2\lambda}[1-z\rho\bar{\rho}] ~& \frac{\lambda+z+\frac{1}{2}q_0^2}{\bar{q}_-}\rho\e^{-2\ii t\theta}
\end{pmatrix},\qquad~~~\qquad z\in\gamma,\\
J^{(1)}_4&=\begin{pmatrix}
\frac{\lambda+z+\frac{1}{2}q_0^2}{q_-}\bar{\rho}\e^{2\ii t\theta} ~& \frac{q_-\delta^2}{2\lambda}[1-z\rho\bar{\rho}]\\[4pt]
-\frac{2\lambda\delta^{-2}}{q_-} ~& -\frac{\lambda-(z+\frac{1}{2}q_0^2)}{\bar{q}_-}\rho\e^{-2\ii t\theta}
\end{pmatrix},\qquad \qquad~~~z\in\bar{\gamma}.
\end{aligned}
\eea
The new jump matrix $J^{(1)}(x,t;z)$ can be analytically extended from $\Sigma^{(1)}$. This leads to the next transformation:
\be
m^{(2)}(x,t;z)=m^{(1)}(x,t;z)H(z),
\ee
where $H(z)$ is defined by
\berr
H(z)=\left\{
\begin{aligned}
&\begin{pmatrix}
d^{\frac{1}{2}} ~& 0\\[4pt]
-z\rho\e^{-2\ii t\theta}\delta^{-2}d^{-\frac{1}{2}} ~& d^{-\frac{1}{2}}
\end{pmatrix},\quad z\in D_1,\\
&\begin{pmatrix}
d^{\frac{1}{2}} ~ & \frac{\bar{\rho}\e^{2\ii t\theta}\delta^2}{1-z\rho\bar{\rho}}d^{\frac{1}{2}}\\[4pt]
0~& d^{-\frac{1}{2}}
\end{pmatrix},\qquad\qquad\quad z\in D_2,\\
&\begin{pmatrix}
d^{-\frac{1}{2}} ~& 0\\[4pt]
\frac{z\rho\e^{-2\ii t\theta}\delta^{-2}}{1-z\rho\bar{\rho}}d^{\frac{1}{2}} ~& d^{\frac{1}{2}}
\end{pmatrix},\qquad\quad~~ z\in D_3,\\
&\begin{pmatrix}
d^{-\frac{1}{2}} ~ & -\bar{\rho}\e^{2\ii t\theta}\delta^2d^{-\frac{1}{2}}\\[4pt]
0 ~& d^{\frac{1}{2}}
\end{pmatrix},\qquad\quad z\in D_4,\\
& I,\qquad\qquad\qquad~\qquad\qquad\qquad z\in D_5\cup D_6.
\end{aligned}
\right.
\eerr
The domains $\{D_j\}_1^6$ are shown on Fig. \ref{fig3}. Then the new function $m^{(2)}(x,t;z)$ solves the following equivalent RH problem:
\bea
\begin{aligned}
&m^{(2)+}(x,t;z)=m^{(2)-}(x,t;z)J^{(2)}(x,t;z),\\
&z\in\Sigma^{(2)}=L_1\cup L_2\cup L_3\cup L_4\cup\gamma\cup\bar{\gamma},
\end{aligned}
\eea
where the curves $L_1,\ldots,L_4$ can be chosen freely, only respecting that they pass through $z_-$, do not cross $\gamma\cup\bar{\gamma}$, go to $\infty$, and lie entirely in domains with the appropriate sign of Re$(\ii\theta)$, and
the jump matrix $J^{(2)}=(H^-)^{-1}(z)J^{(1)}H^+(z)$ is given by
\bea\label{3.13}
\begin{aligned}
J^{(2)}_1&=\begin{pmatrix}
d^{-\frac{1}{2}} ~& 0\\[4pt]
z\rho\e^{-2\ii t\theta}\delta^{-2}d^{-\frac{1}{2}} ~& d^{\frac{1}{2}}
\end{pmatrix},~~ z\in L_1,\quad
J^{(2)}_2=\begin{pmatrix}
d^{-\frac{1}{2}} ~ & -\frac{\bar{\rho}\e^{2\ii t\theta}\delta^2}{1-z\rho\bar{\rho}}d^{\frac{1}{2}}\\[4pt]
0~& d^{\frac{1}{2}}
\end{pmatrix},~~~ z\in L_2,\\
J^{(2)}_3&=\begin{pmatrix}
d^{-\frac{1}{2}} ~& 0\\[4pt]
\frac{z\rho\e^{-2\ii t\theta}\delta^{-2}}{1-z\rho\bar{\rho}}d^{\frac{1}{2}} ~& d^{\frac{1}{2}}
\end{pmatrix},~~~~~ z\in L_3,\quad
J^{(2)}_4=\begin{pmatrix}
d^{-\frac{1}{2}} ~ & -\bar{\rho}\e^{2\ii t\theta}\delta^2d^{-\frac{1}{2}}\\[4pt]
0 ~& d^{\frac{1}{2}}
\end{pmatrix},~z\in L_4,\\
J^{(2)}_5&=\begin{pmatrix}
0 ~& \frac{q_-\delta^2}{\ii\sqrt{z}q_0}\\[4pt]
-\frac{\ii\sqrt{z}q_0}{q_-\delta^2} ~& 0
\end{pmatrix},~~~\qquad z\in \gamma\cup\bar{\gamma}.
\end{aligned}
\eea
We have used the jump condition
\be
\rho^+(z)=\frac{\bar{q}_-}{q_-}\bar{\rho}(z)
\ee
for $\rho(z)$ across $\gamma\cup\bar{\gamma}$ to derive the jump matrix $J^{(2)}_5$.
\begin{figure}[htbp]
  \centering
  \includegraphics[width=3in]{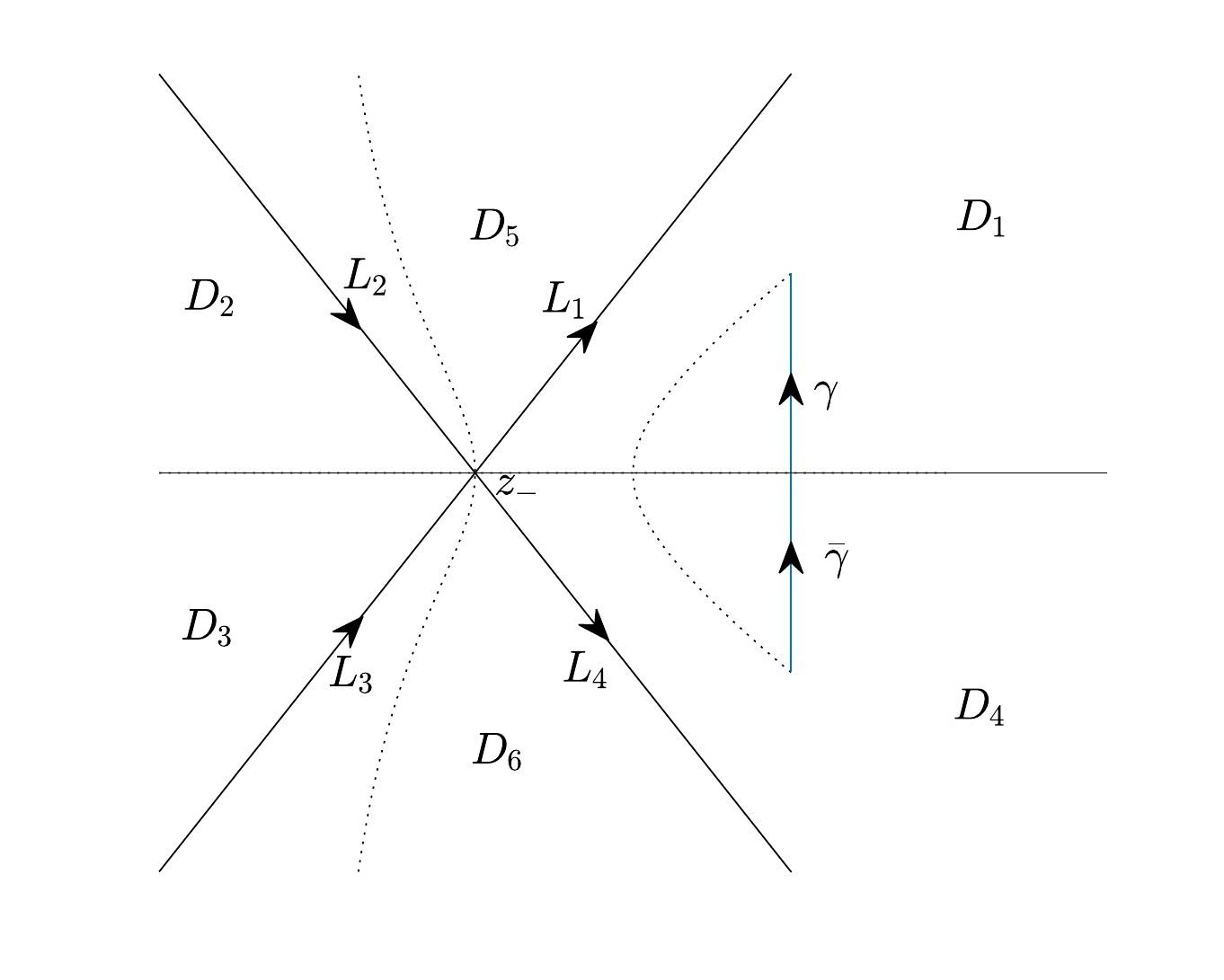}
  \caption{The oriented contour $\Sigma^{(2)}$ and the open sets $\{D_j\}_1^6$ in the complex $z$-plane.}\label{fig3}
\end{figure}

The next deformation is aim to remove the function $d$ from the jump matrices \eqref{3.13} so that the jumps along the contours $L_j$ eventually tend to the identity matrix as $t\rightarrow\infty$. This can be accomplished by letting
\be
m^{(3)}(x,t;z)=m^{(2)}(x,t;z)\hat{H}(z),
\ee
where $\hat{H}(z)$ is defined by
\berr
\hat{H}(z)=\left\{
\begin{aligned}
&I,~\qquad z\in D_j,~j=1,2,3,4,\\
&d^{\frac{\sigma_3}{2}},~\quad z\in D_5,\\
&d^{-\frac{\sigma_3}{2}},\quad z\in D_6.
\end{aligned}
\right.
\eerr
Then it follows that $m^{(3)}(x,t;z)$ satisfies the following jump conditions
\be
m^{(3)+}(x,t;z)=m^{(3)-}(x,t;z)J^{(3)}(x,t;z),\quad z\in\Sigma^{(3)}=\Sigma^{(2)},
\ee
and the jump matrix $J^{(3)}$ is given by
\bea\label{3.18}
\begin{aligned}
J^{(3)}_1&=\begin{pmatrix}
1 ~& 0\\[4pt]
z\rho\e^{-2\ii t\theta}\delta^{-2}  ~& 1
\end{pmatrix},~~ z\in L_1,\quad
J^{(3)}_2=\begin{pmatrix}
1 ~ & -\frac{\bar{\rho}\e^{2\ii t\theta}\delta^2}{1-z\rho\bar{\rho}} \\[4pt]
0~& 1
\end{pmatrix},~~~ z\in L_2,\\
J^{(3)}_3&=\begin{pmatrix}
1 ~& 0\\[4pt]
\frac{z\rho\e^{-2\ii t\theta}\delta^{-2}}{1-z\rho\bar{\rho}}  ~& 1
\end{pmatrix},~~~~~ z\in L_3,\quad
J^{(3)}_4=\begin{pmatrix}
1 ~ & -\bar{\rho}\e^{2\ii t\theta}\delta^2 \\[4pt]
0 ~& 1
\end{pmatrix},~z\in L_4,\\
J^{(3)}_5&=J^{(2)}_5=\begin{pmatrix}
0 ~& \frac{q_-\delta^2}{\ii\sqrt{z}q_0}\\[4pt]
-\frac{\ii\sqrt{z}q_0}{q_-\delta^2} ~& 0
\end{pmatrix},~~~  z\in \gamma\cup\bar{\gamma}.
\end{aligned}
\eea
Obviously, for $z\in L_j$, $j=1,\ldots,4$, the jump matrix $J^{(3)}(x,t;z)$ decays to identity matrix $I$ as $t\rightarrow\infty$ exponentially fast and uniformly outside any neighborhood of $z=z_-$. In order to arrive at a RH problem whose jump matrix does not depend on $z$, we introduce a factorization involving a scalar function $F(z)$ to be defined:
\be\label{3.19}
J^{(3)}_5(x,t;z)=\begin{pmatrix}
(F^-)^{-1}(z) ~& 0\\[4pt]
0 ~& F^-(z)
\end{pmatrix}\begin{pmatrix}
0 ~& \frac{2}{\bar{q}_-}\\[4pt]
-\frac{\bar{q}_-}{2} ~& 0
\end{pmatrix}\begin{pmatrix}
F^+(z) ~& 0\\[4pt]
0 ~& (F^+)^{-1}(z)
\end{pmatrix}
\ee
in such a way that the boundary values $F^\pm(z)$ of $F(z)$ along the two sides of $\gamma\cup\bar{\gamma}$ satisfy
\be\label{3.20}
F^+(z)F^-(z)=\frac{2\ii\sqrt{z}}{q_0}\delta^{-2}(z),\quad z\in\gamma\cup\bar{\gamma}.
\ee
Indeed, once \eqref{3.19} is satisfied, one can absorb the diagonal factors into a new piecewise analytic function whose jump across $\gamma\cup\bar{\gamma}$ is only the constant middle factor in \eqref{3.19}.

The function $F(z)$ can be determined explicitly as follows. Dividing condition \eqref{3.20} by $\lambda^+(z)$, we deduce that
\be
\bigg[\frac{\ln F(z)}{\lambda(z)}\bigg]^+-\bigg[\frac{\ln F(z)}{\lambda(z)}\bigg]^-=\frac{\ln[\frac{2\ii\sqrt{z}}{q_0}\delta^{-2}(z)]}{\lambda^+(z)},\quad z\in\gamma\cup\bar{\gamma},
\ee
and
\berr
\frac{\ln F(z)}{\lambda(z)}=O(z^{-1}),\quad z\rightarrow\infty.
\eerr
Plemelj's formula \cite{MJA} then yields $F(z)$ in the explicit form
\be
F(z)=\exp\bigg\{\frac{\lambda(z)}{2\pi\ii}\int_{\gamma\cup\bar{\gamma}}\frac{\ln[\frac{2\ii\sqrt{s}}{q_0}\delta^{-2}(s)]}{s-z}
\frac{\dd s}{\lambda^+(s)}\bigg\},\quad z\notin\gamma\cup\bar{\gamma}.
\ee
As $z\rightarrow\infty$, we find that
\be
F(\infty)=\e^{\ii\phi(\xi)},
\ee
where
\be\label{3.25}
\phi(\xi)=\frac{1}{2\pi}\int_{\gamma\cup\bar{\gamma}}\bigg[\frac{\ln[\frac{2\ii\sqrt{s}}{q_0}]}{\lambda^+(s)}+
\frac{\ii}{\pi\lambda^+(s)}\int_{-\infty}^{z_-}\frac{\ln[1-u\rho(u)\bar{\rho}(u)]}{u-s}\dd u\bigg]\dd s.
\ee
The factorization \eqref{3.19} suggests the final transformation
\be
m^{(4)}(x,t;z)=F^{\sigma_3}(\infty)m^{(3)}(x,t;z)F^{-\sigma_3}(z).
\ee
Then we have
\be
m^{(4)+}(x,t;z)=m^{(4)-}(x,t;z)J^{(4)}(x,t;z),
\ee
for $z\in\Sigma^{(4)}=\Sigma^{(3)}=L_1\cup L_2\cup L_3\cup L_4\cup\gamma\cup\bar{\gamma}$ (see Fig. \ref{fig3}). Meanwhile, the jump matrix $J^{(4)}=(F^-)^{\sigma_3}J^{(3)}(F^+)^{-\sigma_3}$ satisfies:

$\bullet$ for $z\in\gamma\cup\bar{\gamma}$, jump matrix $J^{(4)}(x,t;z)$ is a constant
\berr
J^{(4)}(x,t;z)=J^{mod}=\begin{pmatrix}
0 ~& \frac{2}{\bar{q}_-}\\[4pt]
-\frac{\bar{q}_-}{2} ~& 0
\end{pmatrix};
\eerr

$\bullet$ for $z\in L_j,~j=1,\ldots,4$, jump matrix $J^{(4)}(x,t;z)$ decays to the identity
\berr
J^{(4)}(x,t;z)=I+O(\e^{-\epsilon t})
\eerr
uniformly outside any neighborhood of $z=z_-$.

Finally, one can write $m^{(4)}$ in the form
\be
m^{(4)}(x,t;z)=m^{err}(x,t;z)m^{mod}(x,t;z),
\ee
where $m^{mod}(x,t;z)$ solves the model problem:
\be
m^{mod+}(x,t;z)=m^{mod-}(x,t;z)J^{mod},\quad z\in\gamma\cup\bar{\gamma}
\ee
with constant jump matrix
\berr
J^{mod}=\begin{pmatrix}
0 ~& \frac{2}{\bar{q}_-}\\[4pt]
-\frac{\bar{q}_-}{2} ~& 0
\end{pmatrix},
\eerr
and
\be\label{3.33}
m^{err}(x,t;z)=I+O(t^{-\frac{1}{2}}).
\ee
The last estimate \eqref{3.33} can be justified by considering the parametrix associated with the RH problem for $m^{(4)}(x,t;z)$, see \cite{RB,GB1}. The error of order $O(t^{-1/2})$ comes from the contribution of the jump near $z_-$.

Define
\be
\nu(z)=\bigg(\frac{z-\frac{\ii}{2}q_0^2}{z+\frac{\ii}{2}q_0^2}\bigg)^{\frac{1}{4}}.
\ee
Then, we have
\berr
\nu^+(z)=\ii\nu^-(z)
\eerr
on $\gamma\cup\bar{\gamma}$, and $\nu(z)$ admits the large-$z$ expansion
\berr
\nu(z)=1-\frac{\ii q_0^2}{4z}+O(z^{-2}),\quad z\rightarrow\infty.
\eerr
As for the model RH problem, its solution thus can be given explicitly in terms of $\nu(z)$:
\be
m^{mod}(x,t;z)=\frac{1}{2}\begin{pmatrix}
\nu(z)+\frac{1}{\nu(z)} ~& \frac{2}{\ii\bar{q}_-}\big(\nu(z)-\frac{1}{\nu(z)}\big)\\[4pt]
\frac{\ii\bar{q}_-}{2}\big(\nu(z)-\frac{1}{\nu(z)}\big) ~& \nu(z)+\frac{1}{\nu(z)}
\end{pmatrix}.
\ee
Then, going back to the determination of $q(x,t)$ in terms of the solution of the basic RH problem, we have
\bea
q(x,t)&=&-2\lim_{z\rightarrow\infty}(z m(x,t;z))_{12}\nn\\
&=&-2\lim_{z\rightarrow\infty}(z m^{(4)}(x,t;z))_{12}F^{-2}(\infty)+O(t^{-\frac{1}{2}})\\
&=&-2\lim_{z\rightarrow\infty}(z m^{mod}(x,t;z))_{12}F^{-2}(\infty)+O(t^{-\frac{1}{2}}).\nn
\eea
Taking into account that $-2\lim_{z\rightarrow\infty}(z m^{mod}(x,t;z))_{12}=q_-$, and $F^{-2}(\infty)=\e^{-2\ii\phi(\xi)}$, we arrive at the following theorem.
\begin{theorem}
In the region $x<-2\sqrt{2}q_0^2t$, as $t\rightarrow\infty$, the asymptotics of the solution $q(x,t)$ of the initial value problem for the GI-type derivative NLS equation \eqref{1.2} takes the form of a plane wave:
\be\label{3.39}
q(x,t)=q_-\e^{-2\ii\phi(\xi)}+O(t^{-\frac{1}{2}}),\quad t\rightarrow\infty,
\ee
where $\phi(\xi)$ is defined by \eqref{3.25}.
\end{theorem}
\begin{remark}
If we let $\xi\rightarrow-\infty$, then $z_-\rightarrow-\infty$. Thus, we get $\phi(\xi)\rightarrow\phi=\frac{1}{2\pi}\int_{\gamma\cup\bar{\gamma}}\frac{\ln[\frac{2\ii\sqrt{s}}{q_0}]}{\lambda^+(s)}
\dd s$, and then the asymptotic formulae \eqref{3.39} reduces to $q(x,t)=q_-\e^{-2\ii\phi}$. This is correspondence to our initial condition up to a phase shift as $x\rightarrow-\infty$.
\end{remark}
\subsection{The modulated elliptic wave region I: $-2\sqrt{2}q_0^2t<x<0$}\label{sec3.2}
In this subsection we compute the leading-order long-time asymptotics of the solution of the GI-type derivative NLS equation \eqref{1.2} in the region $-2\sqrt{2}q_0^2t<x<0$. Recalling the sign structure of Re$(\ii\theta)$ in this region depicted in Fig. \ref{fig2}, the main difference is the absence of real stationary points compared with the plane wave region I. This implies that it is not possible anymore to use the previous factorizations and deformations to lift the contours off the real $z$-axis in such a way that the corresponding jump matrices remain bounded as $t\rightarrow\infty$. Therefore, developing and extending the ideas used in \cite{RB,GB1}, we introduce a new $g$-function $g(z)$ appropriate for the region under consideration, which has a new real stationary point $z_0\in\bfR^-$.

Let us perform the same transformations
\berr
m(x,t;z)\rightsquigarrow m^{(1)}(x,t;z)\rightsquigarrow m^{(2)}(x,t;z)\rightsquigarrow m^{(3)}(x,t;z)
\eerr
as in Subsection \ref{sec3.1} for the plane wave region I but with $\delta(z)$ characterized by $z_0$ instead of the point $z_-$, that is,
\be\label{3.40}
\delta(z)=\exp\bigg\{\frac{1}{2\pi\ii}\int_{-\infty}^{z_0}\frac{\ln[1-s\rho(s)\bar{\rho}(s)]}{s-z}\dd s\bigg\},\quad z\in\bfC\setminus(-\infty,z_0].
\ee
In this subsection, we will use the notation $L_5$ to denote $\gamma\cup\bar{\gamma}$ for simplicity. It then follows that the function $m^{(3)}(x,t;z)$ is analytic in $\bfC\setminus(\bigcup_{j=1}^5L_j)$ and satisfies the jump conditions
\be
m^{(3)+}(x,t;z)=m^{(3)-}(x,t;z)J_j^{(3)}(x,t;z),\quad z\in L_j,~ j=1,\ldots,5,
\ee
with the normalization condition
\be
m^{(3)}(x,t;z)=I+O(z^{-1}),\quad z\rightarrow\infty.
\ee
The jump contours $L_j$ is shown in Fig. \ref{fig4}, the jump matrices $J_j^{(3)}$ are given by equations \eqref{3.18} and $\delta(z)$ defined by \eqref{3.40}.
\begin{figure}[htbp]
  \centering
  \includegraphics[width=3in]{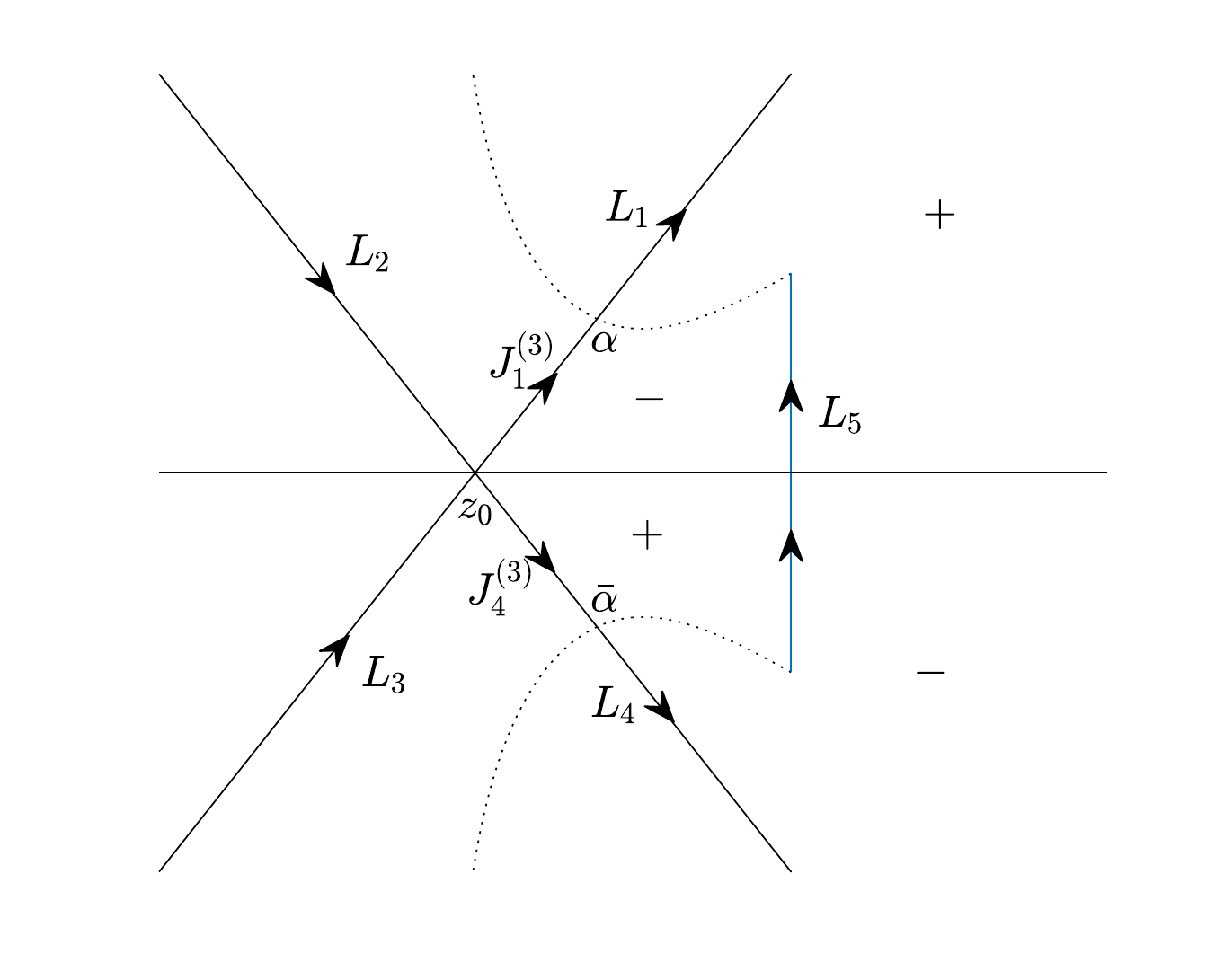}
  \caption{The oriented contours $\{L_j\}_1^5$ in the complex $z$-plane.}\label{fig4}
\end{figure}

We now encounter a new phenomenon that is not present in the plane wave region I, namely, the jumps $J^{(3)}_1$ and $J^{(3)}_4$ grow exponentially with $t$ along the segments $[z_0,\alpha]$ and $[z_0,\bar{\alpha}]$ as shown in Fig. \ref{fig4}, respectively.
In order to overcome this difficulty, we employ the new factorizations for these jumps:
\bea
J^{(3)}_1=J^{(3)}_8J^{(3)}_6J^{(3)}_8,\quad J^{(3)}_4=J^{(3)}_9J^{(3)}_7J^{(3)}_9,
\eea
where
\berr
\begin{aligned}
J^{(3)}_6&=\begin{pmatrix}
0 ~& -\frac{\delta^2}{z\rho}\e^{2\ii t\theta}\\[4pt]
\frac{z\rho}{\delta^2}\e^{-2\ii t\theta} ~& 0
\end{pmatrix},\quad J^{(3)}_7=\begin{pmatrix}
0 ~& -\bar{\rho}\delta^2\e^{2\ii t\theta}\\[4pt]
\frac{\e^{-2\ii t\theta}}{\bar{\rho}\delta^2} ~& 0
\end{pmatrix},\\
J^{(3)}_8&=\begin{pmatrix}
1 ~& \frac{\delta^2}{z\rho}\e^{2\ii t\theta}\\[4pt]
0 ~& 1
\end{pmatrix},~~\quad\quad\qquad J^{(3)}_9=\begin{pmatrix}
1 ~& 0\\[4pt]
-\frac{\e^{-2\ii t\theta}}{\bar{\rho}\delta^2} ~& 1
\end{pmatrix}.
\end{aligned}
\eerr
These factorizations allow for the deformation of the segments $[z_0,\alpha]$ and $[z_0,\bar{\alpha}]$ of Fig. \ref{fig4} into the contours shown in Fig. \ref{fig5}, where $J^{(3)}_j$ $(j=6,\ldots,9)$ denote the restriction of $J^{(3)}$ to the contours labeled by $L_j$ in Fig. \ref{fig5}. We next employ the $g$-function mechanism, more precisely, let
\be\label{3.45}
m^{(4)}(x,t;z)=\e^{\ii G(\infty)t\sigma_3}m^{(3)}(x,t;z)\e^{-\ii G(z)t\sigma_3}
\ee
for a function $G(z)$ that is required to be analytic in $\bfC\setminus(L_5\cup L_6\cup L_7)$. In fact, it is more convenient to consider the function $g(z)$ define by
\be\label{3.46}
g(z)=\theta(\xi;z)+G(z).
\ee
Then, one can infer that $g(z)$ is analytic in $\bfC\setminus(L_5\cup L_6\cup L_7)$ and has jump discontinuities across $L_5$ and $L_6\cup L_7$. Furthermore, the jump conditions of function $m^{(4)}(x,t;z)$ read
\berr
\begin{aligned}
J_1^{(4)}&=\begin{pmatrix}
1 ~& 0\\[4pt]
z\rho\delta^{-2}\e^{-2\ii gt} ~& 1
\end{pmatrix},\quad J_2^{(4)}=\begin{pmatrix}
1 ~& -\frac{\bar{\rho}\delta^2}{1-z\rho\bar{\rho}}\e^{2\ii gt}\\[4pt]
0 ~& 1
\end{pmatrix},\quad
J_3^{(4)}=\begin{pmatrix}
1 ~& 0\\[4pt]
\frac{z\rho\delta^{-2}}{1-z\rho\bar{\rho}}\e^{-2\ii gt} ~& 0
\end{pmatrix},\\
J_4^{(4)}&=\begin{pmatrix}
1 ~& -\bar{\rho}\delta^2\e^{2\ii gt}\\[4pt]
0 ~& 1
\end{pmatrix},\quad
J_5^{(4)}=\begin{pmatrix}
0 ~& \frac{q_-\delta^2}{\ii\sqrt{z}q_0}\e^{\ii(g^++g^-)t}\\[4pt]
-\frac{\ii\sqrt{z}q_0}{q_-\delta^2}\e^{-\ii(g^++g^-)t} ~& 0
\end{pmatrix},\\
J_6^{(4)}&=\begin{pmatrix}
0 ~& -\frac{\delta^2}{z\rho}\e^{\ii (g^++g^-)t}\\[4pt]
\frac{z\rho}{\delta^2}\e^{-\ii (g^++g^-)t} ~& 0
\end{pmatrix},\quad J_7^{(4)}=\begin{pmatrix}
0 ~& -\bar{\rho}\delta^2\e^{\ii (g^++g^-)t}\\[4pt]
\frac{\e^{-\ii (g^++g^-)t}}{\bar{\rho}\delta^2} ~& 0
\end{pmatrix},\\
J_8^{(4)}&=\begin{pmatrix}
1 ~& \frac{\delta^2}{z\rho}\e^{2\ii gt}\\[4pt]
0 ~& 1
\end{pmatrix},\qquad J_9^{(4)}=\begin{pmatrix}
1 ~& 0\\[4pt]
-\frac{\e^{-2\ii gt}}{\bar{\rho}\delta^2} ~& 1
\end{pmatrix}.
\end{aligned}
\eerr
\begin{figure}[htbp]
  \centering
  \includegraphics[width=3in]{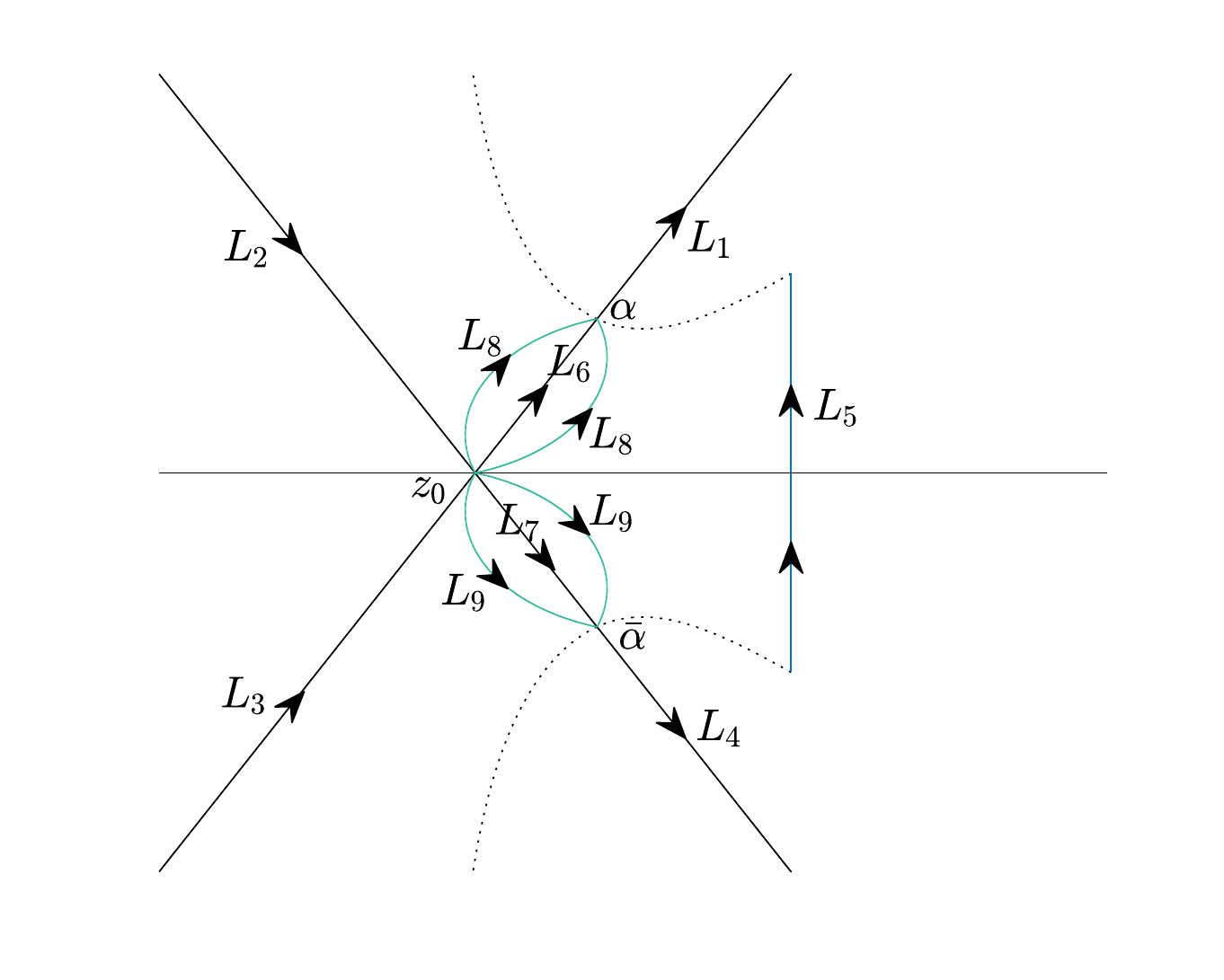}
  \caption{The new oriented contours $\{L_j\}_1^9$ in the complex $z$-plane.}\label{fig5}
\end{figure}

We next turn to determine $z_0$, $\alpha=\alpha_1+\ii\alpha_2$ and $g(z)$ such that all these jump matrices remain bounded as $t\rightarrow\infty$ by using the methods in \cite{GB1,AB3}. Let $$B=L_6\cup(-L_7).$$ Define function $w(z)$ with branch cuts $L_5$ and $B$ by
\be
w(z)=\sqrt{(z^2+\frac{1}{4}q_0^4)(z-\alpha)(z-\bar{\alpha})}.
\ee
We fix the branch cut $B$ to be oriented upwards, then $w(z)=w^-(z)=-w^+(z)$ for $z\in L_5\cup B$. The $w(z)$ gives rise to a genus-1 Riemann surface $\mathcal{X}$ with sheets $\mathcal{X}_1$, $\mathcal{X}_2$ and a basis $\{\textsl{a},\textsl{b}\}$ of cycles defined as follows: the $\textsl{b}$-cycle is a closed, anticlockwise contour around the branch cut $L_5$ that remains entirely on the first sheet $\mathcal{X}_1$ of the Riemann surface; the $\textsl{a}$-cycle consists of an anticlockwise contour that starts on the left side of $B$, then goes $B$ from the right while on the first sheet $\mathcal{X}_1$, and finally returns to the starting point via the second sheet $\mathcal{X}_2$.

Next, we let $g(z)$ be given by the sum of two Abelian integrals:
\berr
g(z)=\frac{1}{2}\bigg(\int_{-\frac{\ii}{2}q_0^2}^z+\int_{\frac{\ii}{2}q_0^2}^z\bigg)\dd g(s),
\eerr
where the Abelian differential $\dd g(z)$ is defined by
\be\label{3.50}
\dd g(z)=-4\frac{(z-z_0)(z-\alpha)(z-\bar{\alpha})}{w(z)}\dd z,
\ee
and $z_0$, $\alpha$ are to be determined. We can also write the Abelian differential $\dd g(z)$ in the following form:
\berr
\dd g(z)=-4\frac{z^3+c_2z^2+c_1z+c_0}{w(z)}\dd z,
\eerr
then, we have
\be\label{3.52}
c_2=-(2\alpha_1+z_0),~c_1=\alpha_1^2+\alpha_2^2+2\alpha_1z_0,~c_0=-z_0(\alpha_1^2+\alpha_2^2).
\ee
We then have
\be
g(z)=-2\bigg(\int_{-\frac{\ii}{2}q_0^2}^z+\int_{\frac{\ii}{2}q_0^2}^z\bigg)\frac{s^3+c_2s^2+c_1s+c_0}{w(s)}\dd s.
\ee
In order to preserve the sign structure in the transition from $\theta$ to $g$, we require that $g(z)$ has the same large-$z$ behavior as the function $\theta(\xi;z)$:
\berr
g(z)=-2z^2+\xi z+O(1), \quad z\rightarrow\infty.
\eerr
This condition implies:
\be\label{3.54}
\begin{aligned}
c_1&=\frac{1}{2}\alpha_2^2+\frac{1}{4}\xi\alpha_1+\frac{1}{8}q_0^4,\\
c_2&=-\frac{\xi}{4}-\alpha_1.
\end{aligned}
\ee
Observe that
\berr
2\bigg(\int_{-\frac{\ii}{2}q_0^2}^z+\int_{\frac{\ii}{2}q_0^2}^z\bigg)\bigg(s-\frac{\xi}{4}\bigg)=2z^2-\xi z+\frac{1}{2}q_0^4.
\eerr
Thus, the large-$z$ asymptotics of $g(z)$ can be specified as
\be\label{3.55}
g(z)=-2z^2+\xi z+g_\infty+O(z^{-1}), \quad z\rightarrow\infty,
\ee
where
\be\label{3.56}
g_\infty=-2\bigg(\int_{-\frac{\ii}{2}q_0^2}^\infty+\int_{\frac{\ii}{2}q_0^2}^\infty\bigg)
\bigg[\frac{s^3+c_2s^2+c_1s+c_0}{w(s)}-\bigg(s-\frac{\xi}{4}\bigg)\bigg]\dd s-\frac{1}{2}q_0^4.
\ee
Moreover, combining \eqref{3.52} and \eqref{3.54}, we get
\be\label{3.57}
\alpha_1=\frac{\xi}{4}-z_0,\quad \alpha_2^2=2z_0^2-\frac{1}{2}\xi z_0+\frac{1}{4}q_0^4.
\ee
It thus remains to determine $z_0$. We do so by analyzing the behavior of $g$ near $\alpha$. In a neighborhood of $\alpha$, it follows from \eqref{3.50} that
\bea
\frac{\dd g}{\dd z}=\frac{-4(\alpha-\bar{\alpha})^{\frac{1}{2}}(\alpha-z_0)}{(\alpha^2+\frac{1}{4}q_0^4)^{\frac{1}{2}}}
\bigg[(z-\alpha)^{\frac{1}{2}}&+&\bigg(\frac{1}{\alpha-z_0}+\frac{1}{2(\alpha-\bar{\alpha})}
-\frac{\alpha}{\alpha^2+\frac{1}{4}q_0^4}\bigg)(z-\alpha)^{\frac{3}{2}}\nn\\
&+&O((z-\alpha)^{\frac{5}{2}})\bigg],\quad z\rightarrow\alpha.\nn
\eea
Thus, we obtain the expansion
\bea
g(z)&=&g(\alpha)-\frac{4(\alpha-\bar{\alpha})^{\frac{1}{2}}(\alpha-z_0)}{(\alpha^2+\frac{1}{4}q_0^4)^{\frac{1}{2}}}
\bigg[\frac{2}{3}(z-\alpha)^{\frac{3}{2}}\nn\\
&&+\frac{2}{5}\bigg(\frac{1}{\alpha-z_0}+\frac{1}{2(\alpha-\bar{\alpha})}
-\frac{\alpha}{\alpha^2+\frac{1}{4}q_0^4}\bigg)(z-\alpha)^{\frac{5}{2}}
+O((z-\alpha)^{\frac{7}{2}})\bigg],\quad z\rightarrow\alpha,\nn
\eea
where
\be
g(\alpha)=-2\bigg(\int_{-\frac{\ii}{2}q_0^2}^\alpha+\int_{\frac{\ii}{2}q_0^2}^{\bar{\alpha}}\bigg)
\frac{s^3+c_2s^2+c_1s+c_0}{w(s)}\dd s-2\int_{\bar{\alpha}}^\alpha\frac{s^3+c_2s^2+c_1s+c_0}{w(s)}\dd s.
\ee
In order to obtain the desired sign structure as shown in Fig. \ref{fig6}, the leading-order term of the expansion of Re$(\ii g)$ near $\alpha$ should be of $O((z-\alpha)^{\frac{3}{2}})$. Thus, we must have Im$g(\alpha)=0$. Hence, we get
\be\label{3.59}
\int_{\bar{\alpha}}^\alpha\frac{s^3+c_2s^2+c_1s+c_0}{w(s)}\dd s=0.
\ee
\begin{figure}[htbp]
  \centering
  \includegraphics[width=3in]{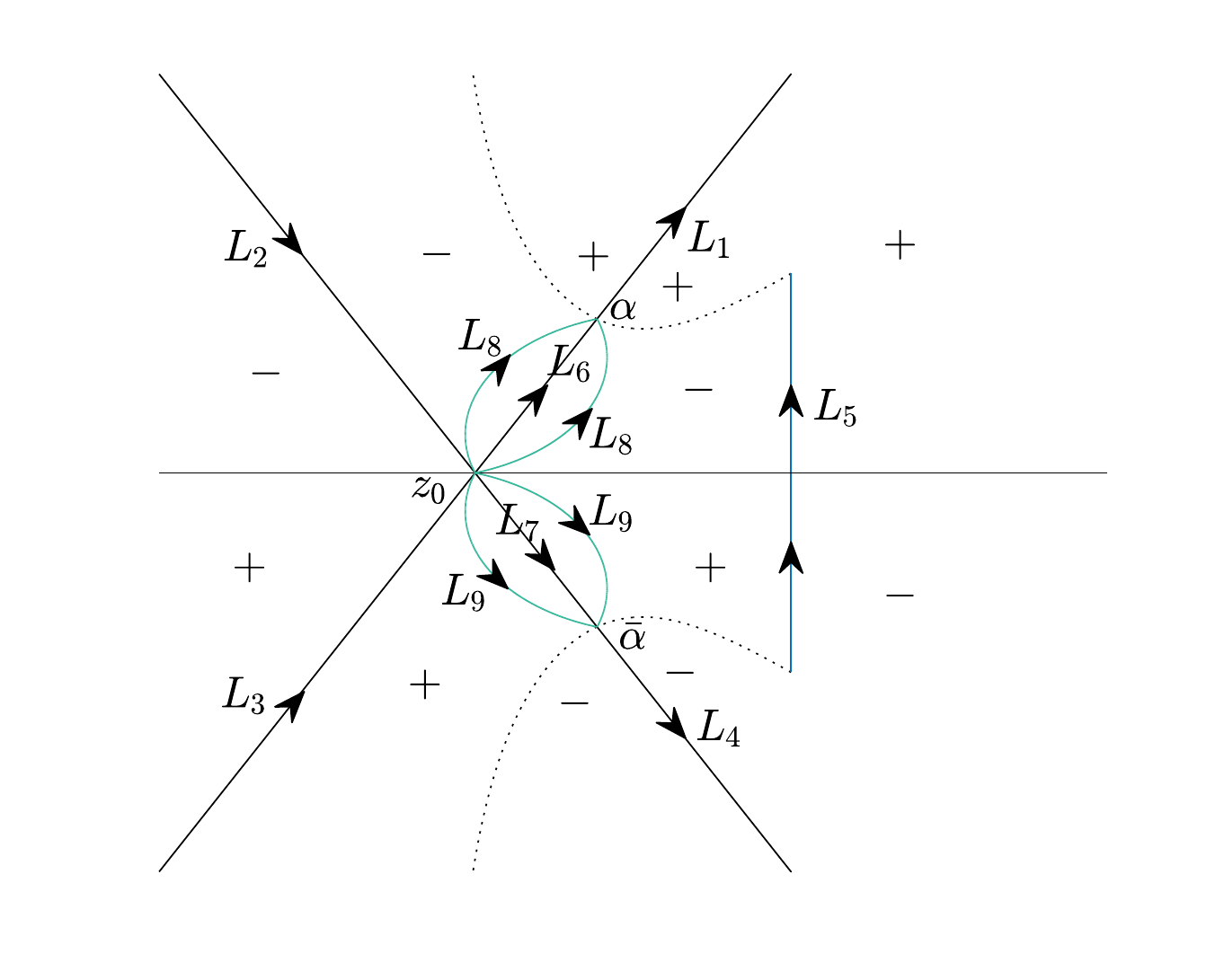}
  \caption{The sign structure of Re(i$g)$ in the complex $z$-plane.}\label{fig6}
\end{figure}
According to the discussion in \cite{GB1}, one can rewrite the condition \eqref{3.59} as the following forms
\be\label{3.60}
\int_{-\frac{\ii}{2}q_0^2}^{\frac{\ii}{2}q_0^2}\frac{s^3+c_2s^2+c_1s+c_0}{w(s)}\dd s=\int_{-\frac{\ii}{2}q_0^2}^{\frac{\ii}{2}q_0^2}\sqrt{\frac{(s-\alpha_1)^2+\alpha_2^2}{s^2+\frac{1}{4}q_0^4}}(s-z_0)\dd s=0.
\ee
Finally, substituting \eqref{3.57} into \eqref{3.60}, we find
\be\label{3.61}
\int_{-\frac{\ii}{2}q_0^2}^{\frac{\ii}{2}q_0^2}\sqrt{\frac{(s-\frac{\xi}{4}+z_0)^2+2z_0^2-\frac{1}{2}\xi z_0+\frac{1}{4}q_0^4}{s^2+\frac{1}{4}q_0^4}}(s-z_0)\dd s=0.
\ee
It is enough to determine the $z_0$ from equation \eqref{3.61}, and hence $\alpha$, function $g(z)$.
\begin{lemma}
For all $\xi\in(-2\sqrt{2}q_0^2,0)$, the integral equation \eqref{3.61} admits a unique solution $z_0=z_0(\xi)$.
\end{lemma}
\begin{proof}
The strategy of the proof follows from \cite{GB1,AB3}. Let
\berr
x=-\frac{\xi}{4q_0^2},\quad y=\frac{4z_0-\xi}{2q_0^2}.
\eerr
Then, equation \eqref{3.61} turns into
\be\label{3.62}
\mathcal{F}(x,y)=\int_{-1}^1\sqrt{\frac{(\ii\tau+y)^2+2y^2-4xy+1}{1-\tau^2}}(\ii\tau+2x-y)\dd\tau=0,
\ee
which is considered for $(x,y)\in[0,\sqrt{2}/2]\times[0,\sqrt{2}]$. It is easy to check that $\mathcal{F}(0,0)=\mathcal{F}(\sqrt{2}/2,\sqrt{2}/2)=0$, $\mathcal{F}_y(0,0)<0$. Moreover, $\mathcal{F}_y(x,y)<0$ for $(x,y)\neq(\sqrt{2}/2,\sqrt{2}/2)$. Therefore, by the implicit function theorem, \eqref{3.62} determines a unique function $y=y(x)$ for any $x\in[0,\sqrt{2}/2]$ such that $y(0)=0$ and $y(\sqrt{2}/2)=\sqrt{2}/2$. That is, for any $\xi\in[-2\sqrt{2}q_0^2,0]$ there exists a unique solution $z_0(\xi)$ of the integral equation \eqref{3.61} with $z_0(-2\sqrt{2}q_0^2)=-(\sqrt{2}/4)q_0^2$.
\end{proof}
We have now specified a $g$-function $g(z)$ which has appropriate signature table as in Fig. \ref{fig6}. Moreover, across the branch cuts $L_5$ and $B$, $g(z)$ satisfies the following jump conditions
\bea\label{3.63}
\begin{aligned}
&g^+(z)+g^-(z)=0,\quad z\in L_5,\\
&g^+(z)+g^-(z)=\Om,\quad z\in L_6\cup L_7,
\end{aligned}
\eea
where the real constant $\Om$ is defined by
\be\label{3.64}
\Om=-4\bigg(\int_{-\frac{\ii}{2}q_0^2}^\alpha+\int_{\frac{\ii}{2}q_0^2}^{\bar{\alpha}}\bigg)
\frac{(z-z_0)(z-\alpha)(z-\bar{\alpha})}{w(z)}\dd z.
\ee

In the following, we continue to deform the RH problem. By the jumps \eqref{3.63} of $g(z)$, we infer that the jump matrices of $m^{(4)}$ can be rewritten as
\berr
\begin{aligned}
J_1^{(4)}&=\begin{pmatrix}
1 ~& 0\\[4pt]
z\rho\delta^{-2}\e^{-2\ii gt} ~& 1
\end{pmatrix},\quad J_2^{(4)}=\begin{pmatrix}
1 ~& -\frac{\bar{\rho}\delta^2}{1-z\rho\bar{\rho}}\e^{2\ii gt}\\[4pt]
0 ~& 1
\end{pmatrix},\quad
J_3^{(4)}=\begin{pmatrix}
1 ~& 0\\[4pt]
\frac{z\rho\delta^{-2}}{1-z\rho\bar{\rho}}\e^{-2\ii gt} ~& 0
\end{pmatrix},\\
J_4^{(4)}&=\begin{pmatrix}
1 ~& -\bar{\rho}\delta^2\e^{2\ii gt}\\[4pt]
0 ~& 1
\end{pmatrix},\quad
J_5^{(4)}=\begin{pmatrix}
0 ~& \frac{q_-\delta^2}{\ii\sqrt{z}q_0}\\[4pt]
-\frac{\ii\sqrt{z}q_0}{q_-\delta^2} ~& 0
\end{pmatrix},\quad
J_6^{(4)}=\begin{pmatrix}
0 ~& -\frac{\delta^2}{z\rho}\e^{\ii\Om t}\\[4pt]
\frac{z\rho}{\delta^2}\e^{-\ii\Om t} ~& 0
\end{pmatrix},\\
 J_7^{(4)}&=\begin{pmatrix}
0 ~& -\bar{\rho}\delta^2\e^{\ii\Om t}\\[4pt]
\frac{\e^{-\ii\Om t}}{\bar{\rho}\delta^2} ~& 0
\end{pmatrix},\quad
J_8^{(4)}=\begin{pmatrix}
1 ~& \frac{\delta^2}{z\rho}\e^{2\ii gt}\\[4pt]
0 ~& 1
\end{pmatrix},\qquad J_9^{(4)}=\begin{pmatrix}
1 ~& 0\\[4pt]
-\frac{\e^{-2\ii gt}}{\bar{\rho}\delta^2} ~& 1
\end{pmatrix},
\end{aligned}
\eerr
where for $j=1,\ldots,9$ we denote by $J^{(4)}_j$  the jump associated with the contour $L_j$. Furthermore, the normalization condition of $m^{(4)}$ is
\be
m^{(4)}(x,t;z)=I+O(z^{-1}),\quad z\rightarrow\infty.
\ee
Recalling \eqref{3.46}, \eqref{3.55}, \eqref{3.56} and
\berr
\theta(\xi;z)=-2z^2+\xi z-\frac{1}{4}q_0^4+O(z^{-1}),\quad z\rightarrow\infty,
\eerr
we find that the real constant $G(\infty)$ involved in the \eqref{3.45} is equal to
\be\label{3.67}
G(\infty)=-2\bigg(\int_{-\frac{\ii}{2}q_0^2}^\infty+\int_{\frac{\ii}{2}q_0^2}^\infty\bigg)
\bigg[\frac{s^3+c_2s^2+c_1s+c_0}{w(s)}-\bigg(s-\frac{\xi}{4}\bigg)\bigg]\dd s-\frac{1}{4}q_0^4.
\ee
Our final task is to eliminate the dependence on $z$ from the jump matrices $J^{(4)}_j$ across the branch cuts $L_j$ for $j=5,6,7$. To achieve this goal, we let
\be\label{3.68}
m^{(5)}(x,t;z)=\e^{-\ii\hat{g}(\infty)\sigma_3}m^{(4)}(x,t;z)\e^{\ii\hat{g}(z)\sigma_3}.
\ee
The function $\hat{g}(z)$ is analytic in $\bfC\setminus\bigcup_{j=5}^7L_j$ with jumps
\be
\begin{aligned}
&\hat{g}^+(z)+\hat{g}^-(z)=\ii\ln\bigg(\frac{2\ii\sqrt{z}}{q_0\delta^2}\bigg),~\quad\qquad z\in L_5,\\
&\hat{g}^+(z)+\hat{g}^-(z)=\ii\ln\bigg(\frac{z\rho}{\delta^2}\bigg)+\omega+\pi,~~ z\in L_6,\\
&\hat{g}^+(z)+\hat{g}^-(z)=-\ii\ln(\bar{\rho}\delta^2)+\omega,~\quad\quad z\in L_7,
\end{aligned}
\ee
with the function $\delta(z)$ defined by \eqref{3.40} and the real constant $\omega$ determined by
\be\label{3.70}
\omega=\ii\frac{\int_{L_5}\frac{\ln\big[\frac{q_0\delta^2(s)}{2\ii\sqrt{s}}\big]}{w(s)}\dd s+\int_{L_6}\frac{\ln\big[\frac{\delta^2(s)}{s\rho(s)}\big]+\ii\pi}{w(s)}\dd s-\int_{L_7}\frac{\ln[\bar{\rho}(s)\delta^2(s)]}{w(s)}\dd s}{(\int_{L_6}-\int_{L_7})\frac{\dd s}{w(s)}}.
\ee
It follows from the Plemelj's formulae that
\be
\hat{g}(z)=\frac{w(z)}{2\pi}\bigg\{\int_{L_5}\frac{\ln\big[\frac{q_0\delta^2(s)}{2\ii\sqrt{s}}\big]}{w(s)(s-z)}\dd s+\int_{L_6}\frac{\ln\big[\frac{\delta^2(s)}{s\rho(s)}\big]+\ii(\omega+\pi)}{w(s)(s-z)}\dd s-\int_{L_7}\frac{\ln[\bar{\rho}(s)\delta^2(s)]+\ii\omega}{w(s)(s-z)}\dd s\bigg\}.
\ee
The definition \eqref{3.70} of $\omega$ ensures that
\berr
\hat{g}(z)=\hat{g}(\infty)+O(z^{-1}),\quad z\rightarrow\infty,
\eerr
where the real constant $\hat{g}(\infty)$ is given by
\be\label{3.72}
\hat{g}(\infty)=-\frac{1}{2\pi}\bigg\{\int_{L_5}\frac{\ln\big[\frac{q_0\delta^2(s)}{2\ii\sqrt{s}}\big]}{w(s)}s\dd s+\int_{L_6}\frac{\ln\big[\frac{\delta^2(s)}{s\rho(s)}\big]+\ii(\omega+\pi)}{w(s)}s\dd s-\int_{L_7}\frac{\ln[\bar{\rho}(s)\delta^2(s)]+\ii\omega}{w(s)}s\dd s\bigg\}.
\ee

Finally, we can obtain the following RH problem for $m^{(5)}(x,t;z)$:
\be
m^{(5)+}(x,t;z)=m^{(5)-}(x,t;z)J^{(5)}(x,t;z),\quad z\in L_j,
\ee
with the jump matrices given by
\berr
\begin{aligned}
J_1^{(5)}&=\begin{pmatrix}
1 ~& 0\\[4pt]
z\rho\delta^{-2}\e^{-2\ii (gt-\hat{g})} ~& 1
\end{pmatrix},\quad J_2^{(5)}=\begin{pmatrix}
1 ~& -\frac{\bar{\rho}\delta^2}{1-z\rho\bar{\rho}}\e^{2\ii(gt-\hat{g})}\\[4pt]
0 ~& 1
\end{pmatrix},\quad
J_3^{(5)}=\begin{pmatrix}
1 ~& 0\\[4pt]
\frac{z\rho\delta^{-2}}{1-z\rho\bar{\rho}}\e^{-2\ii(gt-\hat{g})} ~& 0
\end{pmatrix},\\
J_4^{(5)}&=\begin{pmatrix}
1 ~& -\bar{\rho}\delta^2\e^{2\ii(gt-\hat{g})}\\[4pt]
0 ~& 1
\end{pmatrix},\quad
J_5^{(5)}=\begin{pmatrix}
0 ~& \frac{2}{\bar{q}_-}\\[4pt]
-\frac{\bar{q}_-}{2} ~& 0
\end{pmatrix},\quad
J_6^{(5)}=\begin{pmatrix}
0 ~& \e^{\ii(\Om t-\omega)}\\[4pt]
-\e^{-\ii(\Om t-\omega)} ~& 0
\end{pmatrix},\\
 J_7^{(5)}&=\begin{pmatrix}
0 ~& -\e^{\ii(\Om t-\omega)}\\[4pt]
\e^{-\ii(\Om t-\omega)} ~& 0
\end{pmatrix},\quad
J_8^{(5)}=\begin{pmatrix}
1 ~& \frac{\delta^2}{z\rho}\e^{2\ii(gt-\hat{g})}\\[4pt]
0 ~& 1
\end{pmatrix},\qquad J_9^{(5)}=\begin{pmatrix}
1 ~& 0\\[4pt]
-\frac{\e^{-2\ii(gt-\hat{g})}}{\bar{\rho}\delta^2} ~& 1
\end{pmatrix},
\end{aligned}
\eerr
and the normalization condition
\be
m^{(5)}(x,t;z)=I+O(z^{-1}),\quad z\rightarrow\infty.
\ee
In other words, for the jump matrix $J^{(5)}(x,t;z)$, we have
\berr
J^{(5)}(x,t;z)=\left\{\begin{aligned}
&J^{mod},\qquad~~~~ z\in L_5\cup B,\\
&I+O(\e^{-\epsilon t}),~z\in L_j,~j=1,\ldots,4,8,9,
\end{aligned}
\right.
\eerr
where $B=L_6\cup(-L_7)$ and
\be
J^{mod}=\left\{
\begin{aligned}
&\begin{pmatrix}
0 ~& \frac{2}{\bar{q}_-}\\[4pt]
-\frac{\bar{q}_-}{2} ~& 0
\end{pmatrix},\qquad\quad~\qquad z\in L_5,\\
&\begin{pmatrix}
0 ~& \e^{\ii(\Om t-\omega)}\\[4pt]
-\e^{-\ii(\Om t-\omega)} ~& 0
\end{pmatrix},\quad z\in B.
\end{aligned}
\right.
\ee
Thus, we arrive at the following model problem RH$^{mod}$:
\be\label{3.78}
\begin{aligned}
m^{mod+}(x,t;z)&=m^{mod-}(x,t;z)J^{mod},\quad z\in L_5\cup B,\\
m^{mod}(x,t;z)&=I+O(z^{-1}),~\qquad\qquad z\rightarrow\infty.
\end{aligned}
\ee
The solution of RH$^{mod}$ approximates $m^{(5)}(x,t;z)$ as follows (see the parametrix analysis in \cite{RB})
\be\label{3.79}
m^{(5)}(x,t;z)=\big(I+O(t^{-\frac{1}{2}})\big)m^{mod}(x,t;z).
\ee
The model RH problem \eqref{3.78} can be solved in terms of elliptic theta functions. Let us consider the Abelian differential
\berr
\dd u=\frac{c}{w(z)}\dd z,\quad c=\bigg(\oint_{\textsl{b}}\frac{1}{w(z)}\dd z\bigg)^{-1},
\eerr
which is normalized so that $\oint_{\textsl{b}}\dd u=1$ and has Riemann period $\tau$ defined by
\berr
\tau=\oint_{\textsl{a}}\dd u.
\eerr
Note that $c\in\ii\bfR$. It can be shown that $\tau\in\ii\bfR^+$ (see \cite{HF}). Define
\be
U(z)=\int_{\frac{\ii}{2}q_0^2}^z\dd u.
\ee
Then the following relations are valid:
\be
\begin{aligned}
&U^+(z)+U^-(z)=n,~~~\qquad n\in\bfZ,~z\in L_5,\\
&U^+(z)+U^-(z)=-\tau+n,~~ n\in\bfZ,~z\in B.
\end{aligned}
\ee
Next, we define a new function $\nu(z)$ by
\be
\nu(z)=\bigg(\frac{(z-\frac{\ii}{2}q_0^2)(z-\alpha)}{(z+\frac{\ii}{2}q_0^2)(z-\bar{\alpha})}\bigg)^{\frac{1}{4}},
\ee
which has the same jump discontinuity across both $L_5$ and $B$, namely,
\berr
\nu^+(z)=\ii\nu^-(z),\quad z\in L_5\cup B,
\eerr
and admits the large-$z$ asymptotic behavior
\berr
\nu(z)=1+\frac{q_0^2+2\alpha_2}{4\ii z}+O(z^{-2}),\quad z\rightarrow\infty.
\eerr
The last ingredient is the theta function with $\ii\tau<0$:
\berr
\theta_3(z)=\sum_{l\in\bfZ}\e^{\ii\pi\tau l^2+2\ii\pi lz},
\eerr
which has the following properties
\be
\theta_3(-z)=\theta_3(z),~~\theta_3(z+n)=\theta_3(z),~~\theta(z+\tau)=\e^{-\ii\pi\tau-2\ii\pi z}\theta_3(z).
\ee
Now we define the $2\times2$ matrix-valued function $\Theta(z)=\Theta(x,t;z)$ with entries:
\bea
\Theta_{11}(z)&=&\frac{1}{2}\bigg[\nu(z)+\frac{1}{\nu(z)}\bigg]\frac{\theta_3(U(z)+U_0-\frac{\Om t}{2\pi}+\frac{\omega}{2\pi}+\frac{\ii\ln(\frac{\bar{q}_-}{2})}{2\pi})}{\sqrt{\frac{2}{\bar{q}_-}}\theta_3(U(z)+U_0)},\nn\\
\Theta_{12}(z)&=&\frac{1}{2}\ii\bigg[\nu(z)-\frac{1}{\nu(z)}\bigg]\frac{\theta_3(-U(z)+U_0-\frac{\Om t}{2\pi}+\frac{\omega}{2\pi}+\frac{\ii\ln(\frac{\bar{q}_-}{2})}{2\pi})}{-\sqrt{\frac{\bar{q}_-}{2}}\theta_3(-U(z)+U_0)},\nn\\
\Theta_{21}(z)&=&\frac{1}{2}\ii\bigg[\nu(z)-\frac{1}{\nu(z)}\bigg]\frac{\theta_3(U(z)-U_0-\frac{\Om t}{2\pi}+\frac{\omega}{2\pi}+\frac{\ii\ln(\frac{\bar{q}_-}{2})}{2\pi})}{-\sqrt{\frac{2}{\bar{q}_-}}\theta_3(U(z)-U_0)},\nn\\
\Theta_{22}(z)&=&\frac{1}{2}\bigg[\nu(z)+\frac{1}{\nu(z)}\bigg]\frac{\theta_3(-U(z)-U_0-\frac{\Om t}{2\pi}+\frac{\omega}{2\pi}+\frac{\ii\ln(\frac{\bar{q}_-}{2})}{2\pi})}{-\sqrt{\frac{\bar{q}_-}{2}}\theta_3(-U(z)-U_0)},\nn
\eea
where
\be\label{3.86}
U_0=U(z_*)+\frac{1}{2}(1+\tau),\quad z_*=\frac{q_0^2\alpha_1}{q_0^2+2\alpha_2}.
\ee
Then the solution of the model RH problem \eqref{3.78} is given by
\be
m^{mod}(x,t;z)=\Theta^{-1}(x,t;\infty)\Theta(x,t;z).
\ee
Taking into account \eqref{3.45}, \eqref{3.68} and \eqref{3.79}, we get
\bea
q(x,t)&=&-2\lim_{z\rightarrow\infty}(z m(x,t;z))_{12}\nn\\
&=&-2\lim_{z\rightarrow\infty}(z m^{(4)}(x,t;z))_{12}\e^{-2\ii G(\infty)t}+O(t^{-\frac{1}{2}})\nn\\
&=&-2\lim_{z\rightarrow\infty}(z m^{(5)}(x,t;z))_{12}\e^{2\ii(\hat{g}(\infty)-G(\infty)t)}+O(t^{-\frac{1}{2}})\\
&=&-2\lim_{z\rightarrow\infty}(z m^{mod}(x,t;z))_{12}\e^{2\ii(\hat{g}(\infty)-G(\infty)t)}+O(t^{-\frac{1}{2}}).\nn
\eea
However, we have
\be
-2\lim_{z\rightarrow\infty}(z m^{mod}(x,t;z))_{12}=\bigg(q_-+\frac{2\alpha_2}{\bar{q}_-}\bigg)\frac{\theta_3(-U(\infty)+U_0-\frac{\Om t}{2\pi}+\frac{\omega}{2\pi}+\frac{\ii\ln(\frac{\bar{q}_-}{2})}{2\pi})\theta_3(U(\infty)+U_0)}{\theta_3(U(\infty)+U_0-\frac{\Om t}{2\pi}+\frac{\omega}{2\pi}+\frac{\ii\ln(\frac{\bar{q}_-}{2})}{2\pi})\theta_3(-U(\infty)+U_0)},\nn
\ee
where
\be\label{3.89}
U(\infty)=\int_{\frac{\ii}{2}q_0^2}^\infty\dd u.
\ee
Taking into account \eqref{3.67} and \eqref{3.72}, we get the asymptotics of the solution in the region $-2\sqrt{2}q_0^2<\xi<0$.
\begin{theorem}
In the region $-2\sqrt{2}q_0^2t<x<0$, the asymptotics of the solution $q(x,t)$ of the initial value problem for the GI-type derivative NLS equation \eqref{1.2} takes, as $t\rightarrow\infty$, the form of a modulated elliptic wave:
\bea
q(x,t)&=&\bigg(q_-+\frac{2\alpha_2}{\bar{q}_-}\bigg)\frac{\theta_3(-U(\infty)+U_0-\frac{\Om t}{2\pi}+\frac{\omega}{2\pi}+\frac{\ii\ln(\frac{\bar{q}_-}{2})}{2\pi})\theta_3(U(\infty)+U_0)}{\theta_3(U(\infty)+U_0-\frac{\Om t}{2\pi}+\frac{\omega}{2\pi}+\frac{\ii\ln(\frac{\bar{q}_-}{2})}{2\pi})\theta_3(-U(\infty)+U_0)}
\e^{2\ii(\hat{g}(\infty)-G(\infty)t)}\nn\\
&&+O(t^{-\frac{1}{2}}),\quad t\rightarrow\infty,
\eea
where the constants $\alpha_2$, $\Om$, $\omega$, $U(\infty)$, $U_0$, $G(\infty)$ and $\hat{g}(\infty)$ are given by the equations \eqref{3.57}, \eqref{3.64}, \eqref{3.70}, \eqref{3.89}, \eqref{3.86}, \eqref{3.67} and \eqref{3.72}, respectively.
\end{theorem}

\subsection{The plane wave and modulated elliptic wave regions II}
In the rest of the present paper, we devoted to discuss the long-time asymptotics of the solution of the initial value problem for GI-type derivative NLS equation \eqref{1.2} in the plane wave region II: $x>2\sqrt{2}q_0^2t$ and modulated elliptic wave region II: $0<x<2\sqrt{2}q_0^2t$.

We first consider the region $x>2\sqrt{2}q_0^2t$. In this case we first rescale the RH problem \eqref{2.38} for $m(x,t;z)$ as follows, which is motivated by the ideas used in \cite{GB1}. Let
\be
\tilde{m}(x,t;z)=\left\{\begin{aligned}
&m(x,t;z)A(k),\qquad z\in\bfC^+\setminus\gamma,\\
&m(x,t;z)\bar{A}^{-1}(k),\quad z\in\bfC_-\setminus\bar{\gamma},
\end{aligned}\right.
\ee
where
\berr
A(k)=\begin{pmatrix}
a(k) ~& 0\\[4pt]
0 ~& a^{-1}(k)
\end{pmatrix}.
\eerr
Then, $\tilde{m}(x,t;z)$ is analytic in $\bfC\setminus\Sigma$ and satisfies the following jump conditions
\be
\begin{aligned}
&\tilde{m}^+(x,t;z)=\tilde{m}^-(x,t;z)\tilde{J}^{(0)}_1(x,t;z),\quad z\in\bfR,\\
&\tilde{m}^+(x,t;z)=\tilde{m}^-(x,t;z)\tilde{J}^{(0)}_2(x,t;z),\quad z\in\gamma,\\
&\tilde{m}^+(x,t;z)=\tilde{m}^-(x,t;z)\tilde{J}^{(0)}_3(x,t;z),\quad z\in\bar{\gamma},\\
\end{aligned}
\ee
with the normalization condition
\be
\tilde{m}(x,t;z)=I+O(z^{-1}),\quad z\rightarrow\infty,
\ee
where
\berr
\begin{aligned}
\tilde{J}^{(0)}_1(x,t;z)&=\begin{pmatrix}
\frac{1}{d(z)} ~ & -\bar{\varrho}(z)\e^{2\ii t\theta(\xi;z)}\\[4pt]
z\varrho(z)\e^{-2\ii t\theta(\xi;z)} ~& d(z)[1-z\varrho(z)\bar{\varrho}(z)]\\
\end{pmatrix},\\
\tilde{J}^{(0)}_2(x,t;z)&=\begin{pmatrix}
-\frac{\lambda-(z+\frac{1}{2}q_0^2)}{q_+}\bar{\varrho}(z)\e^{2\ii t\theta(\xi;z)}~& -\frac{2\lambda}{z\bar{q}_+}[1-z\varrho(z)\bar{\varrho}(z)]\\[4pt]
\frac{z\bar{q}_+}{2\lambda} ~& \frac{\lambda+z+\frac{1}{2}q_0^2}{\bar{q}_+}\varrho(z)\e^{-2\ii t\theta(\xi;z)}
\end{pmatrix},\\
\tilde{J}^{(0)}_3(x,t;z)&=\begin{pmatrix}
\frac{\lambda+z+\frac{1}{2}q_0^2}{q_+}\bar{\varrho}(z)\e^{2\ii t\theta(\xi;z)} ~& \frac{q_+}{2\lambda}\\[4pt]
-\frac{2\lambda}{q_+}[1-z\varrho(z)\bar{\varrho}(z)] ~& -\frac{\lambda-(z+\frac{1}{2}q_0^2)}{\bar{q}_+}\varrho(z)\e^{-2\ii t\theta(\xi;z)}
\end{pmatrix},
\end{aligned}
\eerr
and
\berr
\varrho(z)=-\frac{\bar{b}(k)}{k\bar{a}(k)}.
\eerr
Then, we have
\be
\varrho^+(z)=\frac{\bar{q}_+}{q_+}\bar{\varrho}(z),\quad z\in\gamma\cup\bar{\gamma}.
\ee
\begin{figure}[htbp]
  \centering
  \includegraphics[width=3in]{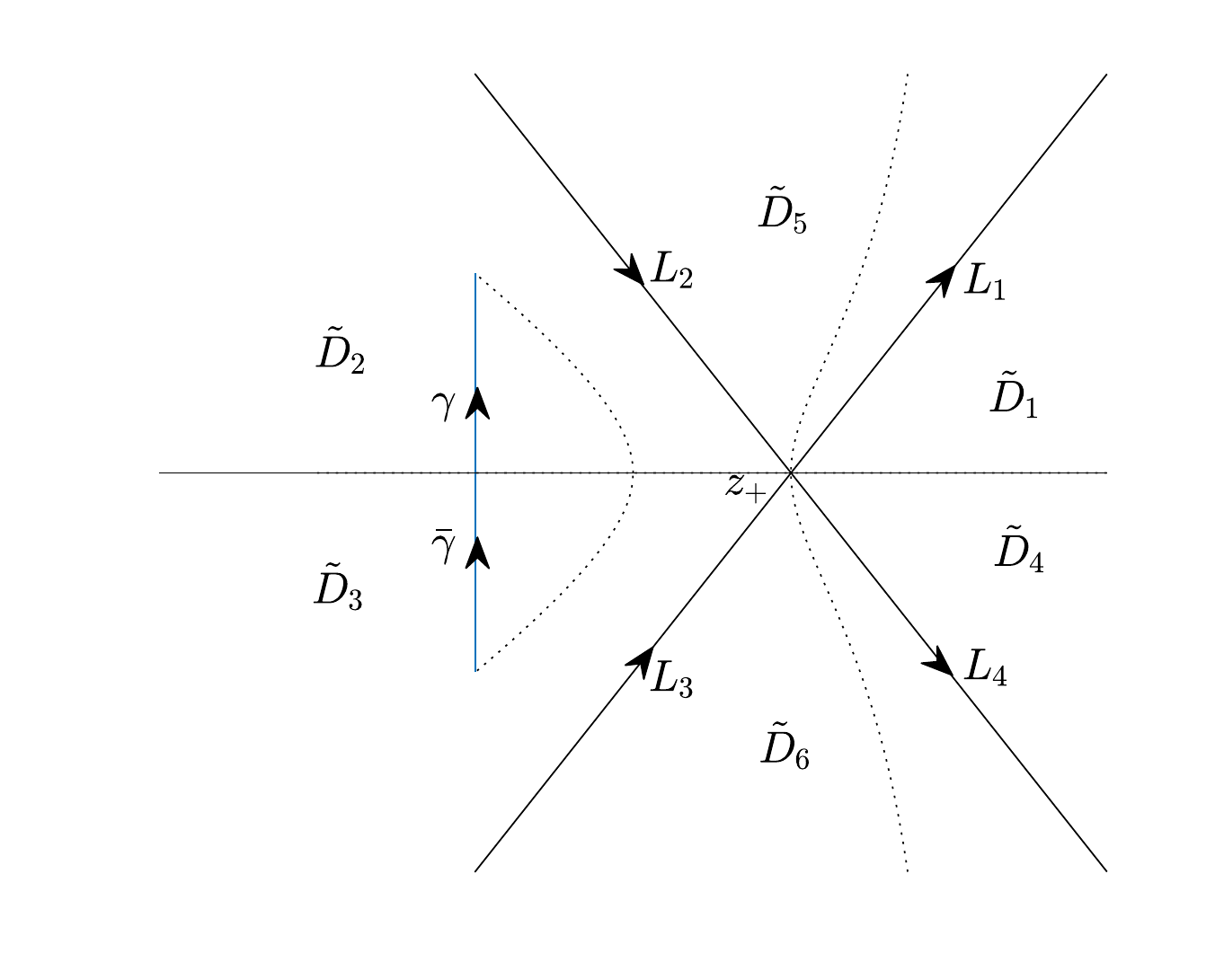}
  \caption{The oriented contour $\Sigma^{(2)}$ and the open sets $\{\tilde{D}_j\}_1^6$ in the plane wave region II.}\label{fig7}
\end{figure}

As the discussions in plane wave region I, the first deformation for $\tilde{m}(x,t;z)$ is
\be
\tilde{m}^{(1)}(x,t;z)=\tilde{m}(x,t;z)\delta^{-\sigma_3}(z),
\ee
where
\be
\delta(z)=\exp\bigg\{\frac{\ii}{2\pi}\int^{\infty}_{z_+}\frac{\ln[1-s\varrho(s)\bar{\varrho}(s)]}{s-z}\dd s\bigg\},\quad z\in\bfC\setminus[z_+,\infty).
\ee
Then $\tilde{m}^{(1)}(x,t;z)$ satisfies the following RH problem
\bea
\tilde{m}^{(1)+}(x,t;z)=\tilde{m}^{(1)-}(x,t;z)\tilde{J}^{(1)}(x,t;z),\quad z\in\Sigma,
\eea
where the jump matrix is given by
\berr
\begin{aligned}
\tilde{J}^{(1)}_1&=\begin{pmatrix}
d^{-\frac{1}{2}} ~ & 0\\[4pt]
z\varrho\e^{-2\ii t\theta}\delta^{-2}d^{\frac{1}{2}} ~& d^{\frac{1}{2}}\\
\end{pmatrix}\begin{pmatrix}
d^{-\frac{1}{2}} ~ & -\bar{\varrho}\e^{2\ii t\theta}\delta^2d^{\frac{1}{2}}\\[4pt]
0 ~& d^{\frac{1}{2}}\\
\end{pmatrix},\qquad~~~~ z\in(-\infty,z_+),\\
\tilde{J}^{(1)}_2&=\begin{pmatrix}
d^{-\frac{1}{2}} ~ & -\frac{\bar{\varrho}\e^{2\ii t\theta}(\delta^-)^2}{1-z\varrho\bar{\varrho}}d^{-\frac{1}{2}}\\[4pt]
0 ~& d^{\frac{1}{2}}\\
\end{pmatrix}\begin{pmatrix}
d^{-\frac{1}{2}} ~ & 0\\[4pt]
\frac{z\varrho\e^{-2\ii t\theta}(\delta^+)^{-2}}{1-z\varrho\bar{\varrho}}d^{-\frac{1}{2}} ~& d^{\frac{1}{2}}\\
\end{pmatrix},\quad z\in(z_+,\infty),\\
\tilde{J}^{(1)}_3&=\begin{pmatrix}
-\frac{\lambda-(z+\frac{1}{2}q_0^2)}{q_+}\bar{\varrho}\e^{2\ii t\theta}~& -\frac{2\lambda\delta^2}{z\bar{q}_+}[1-z\varrho\bar{\varrho}]\\[4pt]
\frac{z\bar{q}_+\delta^{-2}}{2\lambda} ~& \frac{\lambda+z+\frac{1}{2}q_0^2}{\bar{q}_+}\varrho\e^{-2\ii t\theta}
\end{pmatrix},\qquad~~~~~~\qquad z\in\gamma,\\
\tilde{J}^{(1)}_4&=\begin{pmatrix}
\frac{\lambda+z+\frac{1}{2}q_0^2}{q_+}\bar{\varrho}\e^{2\ii t\theta} ~& \frac{q_+\delta^2}{2\lambda}\\[4pt]
-\frac{2\lambda\delta^{-2}}{q_+}[1-z\varrho\bar{\varrho}] ~& -\frac{\lambda-(z+\frac{1}{2}q_0^2)}{\bar{q}_+}\varrho\e^{-2\ii t\theta}
\end{pmatrix},\qquad \qquad~~~z\in\bar{\gamma}.
\end{aligned}
\eerr
The next transformation is
\be
\tilde{m}^{(2)}(x,t;z)=\tilde{m}^{(1)}(x,t;z)\tilde{H}(z),
\ee
where $\tilde{H}(z)$ is defined by
\berr
\tilde{H}(z)=\left\{
\begin{aligned}
&\begin{pmatrix}
d^{\frac{1}{2}} ~ & 0\\[4pt]
-\frac{z\varrho\e^{-2\ii t\theta}\delta^{-2}}{1-z\varrho\bar{\varrho}}d^{-\frac{1}{2}} ~& d^{-\frac{1}{2}}\\
\end{pmatrix},\quad z\in \tilde{D}_1,\\
&\begin{pmatrix}
d^{\frac{1}{2}} ~ & \bar{\varrho}\e^{2\ii t\theta}\delta^2d^{\frac{1}{2}}\\[4pt]
0 ~& d^{-\frac{1}{2}}\\
\end{pmatrix},\qquad\quad\quad z\in \tilde{D}_2,\\
&\begin{pmatrix}
d^{-\frac{1}{2}} ~ & 0\\[4pt]
z\varrho\e^{-2\ii t\theta}\delta^{-2}d^{\frac{1}{2}} ~& d^{\frac{1}{2}}\\
\end{pmatrix},~~\quad\quad z\in \tilde{D}_3,\\
&\begin{pmatrix}
d^{-\frac{1}{2}} ~ & -\frac{\bar{\varrho}\e^{2\ii t\theta}\delta^2}{1-z\varrho\bar{\varrho}}d^{-\frac{1}{2}}\\[4pt]
0 ~& d^{\frac{1}{2}}
\end{pmatrix},~~\quad\quad z\in \tilde{D}_4,\\
& I,\qquad\qquad\qquad\quad\qquad\qquad~~ z\in \tilde{D}_5\cup \tilde{D}_6,
\end{aligned}
\right.
\eerr
and the domains $\{\tilde{D}_j\}_1^6$ are shown in Fig. \ref{fig7}. Then the function $\tilde{m}^{(2)}(x,t;z)$ solves the following equivalent RH problem:
\bea
\begin{aligned}
&\tilde{m}^{(2)+}(x,t;z)=\tilde{m}^{(2)-}(x,t;z)\tilde{J}^{(2)}(x,t;z),\\
&z\in\Sigma^{(2)}=L_1\cup L_2\cup L_3\cup L_4\cup\gamma\cup\bar{\gamma},
\end{aligned}
\eea
where the contours $L_j$ are shown in Fig. \ref{fig7} and the jump matrix $\tilde{J}^{(2)}$ is given by
\berr
\begin{aligned}
\tilde{J}^{(2)}_1&=\begin{pmatrix}
d^{-\frac{1}{2}} ~ & 0\\[4pt]
\frac{z\varrho\e^{-2\ii t\theta}(\delta^+)^{-2}}{1-z\varrho\bar{\varrho}}d^{-\frac{1}{2}} ~& d^{\frac{1}{2}}
\end{pmatrix},~~ z\in L_1,\quad
\tilde{J}^{(2)}_2=\begin{pmatrix}
d^{-\frac{1}{2}} ~ & -\bar{\varrho}\e^{2\ii t\theta}\delta^2d^{\frac{1}{2}}\\[4pt]
0 ~& d^{\frac{1}{2}}\\
\end{pmatrix},~~~ z\in L_2,\\
\tilde{J}^{(2)}_3&=\begin{pmatrix}
d^{-\frac{1}{2}} ~ & 0\\[4pt]
z\varrho\e^{-2\ii t\theta}\delta^{-2}d^{\frac{1}{2}} ~& d^{\frac{1}{2}}
\end{pmatrix},~~~~~ z\in L_3,\quad
\tilde{J}^{(2)}_4=\begin{pmatrix}
d^{-\frac{1}{2}} ~ & -\frac{\bar{\varrho}\e^{2\ii t\theta}(\delta^-)^2}{1-z\varrho\bar{\varrho}}d^{-\frac{1}{2}}\\[4pt]
0 ~& d^{\frac{1}{2}}\\
\end{pmatrix},~z\in L_4,\\
\tilde{J}^{(2)}_5&=\begin{pmatrix}
0 ~& \frac{q_+\delta^2}{\ii\sqrt{z}q_0}\\[4pt]
-\frac{\ii\sqrt{z}q_0}{q_+\delta^2} ~& 0
\end{pmatrix},~~~\qquad z\in \gamma\cup\bar{\gamma}.
\end{aligned}
\eerr
Then, performing the similar deformation and analysis as in Subsection \ref{sec3.1}, we can get a model RH problem which can be solved explicitly and obtain the long-time asymptotics for the solution $q(x,t)$. We just give the main results here.
\begin{theorem}
In the region $x>2\sqrt{2}q_0^2t$, the asymptotics of the solution $q(x,t)$ as $t\rightarrow\infty$ of the initial value problem for the GI-type derivative NLS equation \eqref{1.2} takes the form of a plane wave:
\be
q(x,t)=q_+\e^{2\ii\tilde{\phi}(\xi)}+O(t^{-\frac{1}{2}}),\quad t\rightarrow\infty,
\ee
where $\tilde{\phi}(\xi)$ is given by
\berr
\tilde{\phi}(\xi)=\frac{1}{2\pi}\int_{\gamma\cup\bar{\gamma}}\bigg[\frac{\ln[\frac{2\ii\sqrt{s}}{q_0}]}{\lambda^+(s)}+
\frac{1}{\ii\pi\lambda^+(s)}\int^{\infty}_{z_+}\frac{\ln[1-u\varrho(u)\bar{\varrho}(u)]}{u-s}\dd u\bigg]\dd s.
\eerr
\end{theorem}

For the modulated elliptic wave region II: $0<x<2\sqrt{2}q_0^2t$, the asymptotic analysis is entirely analogous to Subsection \ref{sec3.2} after the rescaling the RH problem for $m(x,t;z)$ as stated above. We omit the detailed computation here.

{\bf Acknowledgments.}

This work was supported in part by the National Natural Science Foundation of
China under grants 11731014, 11571254 and 11471099.
\medskip
\small{

}
\end{document}